\documentclass[12pt]{article}

\usepackage{amsfonts}
\usepackage{mathtools} % pour \eqqcolon
\usepackage{amsmath}
\usepackage{amsthm}
\usepackage{algorithmic}
\usepackage{algorithm}
\usepackage{subfigure}
\usepackage{graphicx}
\newtheorem{remark}{Remark}
\newtheorem{definition}{Definition}
\newtheorem{proposition}{Proposition}
\newtheorem{lemma}{Lemma}

\usepackage[scale=0.75]{geometry}

\newcommand{\norm}[1]{\left \| #1 \right \|}

\newcommand{\Va}{\boldsymbol{a}}
\newcommand{\Vx}{\boldsymbol{x}}

\newcommand{\Vy}{\boldsymbol{y}}
\newcommand{\Vu}{\boldsymbol{u}}

\newcommand{\Ve}{\boldsymbol{e}}
\newcommand{\Vr}{\boldsymbol{r}}
\newcommand{\Vs}{\boldsymbol{s}}

\newcommand{\Vg}{\boldsymbol{g}}
\newcommand{\Valp}{\boldsymbol{\alpha}_r}

\newcommand{\Vf}{\boldsymbol{f}}
\newcommand{\Vh}{\boldsymbol{h}}
\newcommand{\VXi}{\boldsymbol{\Xi}}
\newcommand{\Vxi}{\boldsymbol{\xi}}
\newcommand{\Vv}{\boldsymbol{v}}

\newcommand{\MA}{\mathbf{A}}

\newcommand{\MC}{\mathbf{C}}

\newcommand{\MM}{\mathbf{M}}

\newcommand{\MI}{\mathbf{I}}

\newcommand{\MV}{\mathbf{V}_r}

\newcommand{\MP}{\boldsymbol{\Pi}}
\newcommand{\MPm}{\mathbf{P}_m}
\newcommand{\MHm}{\mathbf{H}_m}
\newcommand{\Ind}{\boldsymbol{1}}

\newcommand{\RR}{\mathbb{R}}

\newcommand{\Exp}{ \mathbb{E}}

\DeclareMathOperator{\spann}{\mathrm{span}}
\title{Dynamical model reduction method for solving parameter-dependent dynamical systems
\thanks{This work was supported by the french research group {GDR MoMaS}.}
}
\author{Marie BILLAUD-FRIESS\footnotemark[2]\footnotemark[3]  and Anthony NOUY\footnotemark[2]}

\begin{document}
\maketitle

\renewcommand{\thefootnote}{\fnsymbol{footnote}}
\footnotetext[2]{Ecole Centrale de Nantes, GeM, UMR CNRS 6183, France.}
\footnotetext[3]{Corresponding author (marie.billaud-friess@ec-nantes.fr).}
\renewcommand{\thefootnote}{\arabic{footnote}}

%\slugger{mms}{xxxx}{xx}{x}{x--x}%slugger should be set to mms, siap, sicomp, sicon, sidma, sima, simax, sinum, siopt, sisc, or sirev

\begin{abstract}
We propose a projection-based model order reduction method 
for the solution of parameter-dependent dynamical systems. The proposed method relies on the construction of time-dependent reduced spaces generated from evaluations of the solution of the full-order model at some selected parameters values. The approximation obtained by Galerkin projection is the solution of a reduced dynamical system with a modified flux which takes into account the  time dependency of the reduced spaces. An a posteriori error estimate is derived  and a greedy algorithm using this error estimate is proposed for the adaptive selection of parameters values. The resulting method can be interpreted as a dynamical low-rank approximation method with a subspace point of view and a uniform control of the error over the parameter set. 
\end{abstract}

%%%%%%%%%%%%%%%%%%%%%%%%%%%%%%%%%%%%%%
% TABLE DES MATIERES
%%%%%%%%%%%%%%%%%%%%%%%%%%%%%%%%%%%%%%

\section{Introduction}\label{sec:intro}

Parameter-dependent equations are considered in many problems of scientific computing such as optimization, control or uncertainty quantification. For complex numerical models, 
model order reduction methods are usually required for an efficient estimation of the solution for many values of the parameters (multi-query context). Classical model order reduction methods for parameter-dependent equations are the Reduced Basis (RB) method \cite{Haasdonk2014}, the  Proper Orthogonal Decomposition  (POD) method \cite{Volkwein2008} or the Proper Generalized Decomposition (PGD) method \cite{Nouy2010}. These methods can be interpreted as  low-rank approximation methods  with different constructions of the approximation for different controls of the error over the parameter set (uniform control for RB or control in mean-square sense for POD and PGD, see, e.g., \cite{Nouy2015}) .

This paper is concerned with the solution of parameter-dependent non-autonomous dynamical systems of the form
 \begin{equation}
\left\{ 
\begin{aligned}
\Vu'(t,\Vxi)&=  \Vf(\Vu(t,\Vxi),t,\Vxi), 
\\
\Vu(0,\Vxi)&= \Vu^0(\Vxi),
\end{aligned}
\right.
\label{eq:dynsystgen}
\end{equation}
where the flux $\Vf$ and initial condition depend on some parameters $\Vxi$ with values in a parameter set $\VXi $. 
%We assume that  $\Vf$ is Lipschitz continuous with respect to its first variable $\Vu$. 
The solution $\Vu(t,\Vxi)$ belongs to the high-dimensional state space $X=\mathbb{R}^d$. 
Regarding time-dependent problems, different model order reduction (MOR) methods have been considered in the literature.  In the context of RB methods, an approximation  is obtained by a (Petrov-)Galerkin projection of the solution onto a time-independent low-dimensional space $X_r$ (so-called reduced space) of $X$, which results in an approximation of the form 
\begin{equation}
\Vu(t,\Vxi) \approx \sum_{i=1}^r \Vv_i \alpha_i(t,\Vxi),
\label{eq:RB}
\end{equation}
where $\{ \Vv_i \}_{i=1}^r$ constitutes a basis of $X_r$. 
Different methods have been proposed for the construction of 
time-independent reduced spaces $X_r$ (see, e.g., \cite{Bui2008,Baur2015} for a review of such methods). In \cite{Grepl2005}, $X_r$ is obtained as the span of snapshots $\Vu(t^k,\Vxi^i)$ (in both time and parameter) of the  solution of the full-order model.  However, for a high-dimensional state space $X$, it is not feasible (and far from optimal) to retain a large number of snapshots. Then, one can rely on a POD of the snapshots matrix in order to extract subspaces which are optimal in a mean-square sense. A more popular approach has been considered in \cite{Eftang2011,Haasdonk2008,Haasdonk2011,Haasdonk2013} for linear evolution problems, with an adaptive construction of the reduced space by a POD-greedy algorithm using a posteriori error estimates. At each iteration of this algorithm, the reduced space is enriched by the dominant modes of the POD of a trajectory $\Vu(\cdot,\Vxi^i)$ of the full-order model, where $\Vxi^i$ maximizes over the parameters set an estimate of the current approximation error.  This strategy has also been considered in \cite{Nguyen2009,Drohmann2012,Janon2013,Wirtz2014} for the solution of nonlinear problems, including nonlinear dynamical systems. In \cite{Wirtz2014}, it was combined with a discrete variant of the Empirical Interpolation Method \cite{Barrault2004,Chaturantabut2010} for the approximation of nonlinear terms  and for the efficient evaluation of the error estimate which is required in the greedy algorithm.

PGD method has been considered in \cite{Nouy2008,Nouy2010} for the low-rank approximation of the solution of stochastic evolution equations, with an approximation of the form 
\begin{align}
\Vu(t,\Vxi) \approx \sum_{i=1}^r \Vv_i(t) \alpha_i(\Vxi),
\label{eq:PGD-STPG}
\end{align}
which is seen as a rank-$r$ element in the tensor space $X^{[0,T]}\otimes \mathbb{R}^{\VXi}$. This approach adopts a variational approach in time. The resulting approximation can be seen as a projection of $\Vu(t,\Vxi) \in X$
onto a time-dependent reduced space 
$$X_r(t) = \mathrm{span}\{\Vv_1(t),\hdots ,\Vv_r(t) \},$$
 which allows to well capture transient phenomena. However, 
the projection is obtained with a global in time variational principle and is not optimal at each instant $t$. 
{In the context of parameter dependent equations, reduced basis methods based on  Petrov-Galerkin space-time (PG-ST) formulations have also been introduced, see for example \cite{Steih2011,Urban2012}. Such an approach provides a low-rank approximation of the form \eqref{eq:PGD-STPG}. 
At the discrete level, it differs from usual reduced basis approaches, introduced e.g. in \cite{Haasdonk2008}, since they do not rely on time-stepping scheme except in very particular cases. For a detailed comparison with standard RB method, see the recent review \cite{Glas2016} .}

Dynamical low-rank approximation methods have been considered in {\cite{Cheng2013-1,Cheng2013-2,Koch2007,Musharbash2014,Sapsis2009}}, with different types of constructions but an approximation of the same form 
\begin{align}
\Vu(t,\Vxi) \approx \sum_{i=1}^r \Vv_i(t) \alpha_i(t,\Vxi), \label{approxformdynamical}
\end{align}
 which at each instant $t$ can be interpreted as a rank-$r$ approximation in the tensor space $X \otimes \mathbb{R}^{\VXi}$. Again, the approximation can be seen as a projection onto a time-dependent reduced space $X_r(t)$ but here, the projection is obtained through principles which are local in time (e.g., Dirac-Frenkel principle). This results in a reduced order model which takes the form of a dynamical system. 
\\

In this paper, we propose a new model order reduction method for solving  
general parameter-dependent dynamical systems of the form \eqref{eq:dynsystgen}. This method provides an approximation of the form \eqref{approxformdynamical}, so that it can be interpreted as a dynamical low-rank approximation method. However, the proposed method differs from existing dynamical low-rank approximation methods in that it adopts a subspace point of view and provides a uniform control of the error over the parameter set. 
%{In that sense, it is closer to a RB method with time-dependent reduced spaces,
%but it remains different from the PG-ST reduced basis strategy since it provides an approximation under the form \eqref{approxformdynamical} in place of \eqref{eq:PGD-STPG}.} 
The reduced space $X_r(t) \subset X$ is here obtained as the span of some selected trajectories $\{\Vu(\cdot,\Vxi^i)\}_{i=1}^r$ of the full-order model  \eqref{eq:dynsystgen}. The approximation is obtained by solving a reduced dynamical system of size $r$ obtained by Galerkin projection (see, e.g., \cite{Benner2013}), which has the form of the initial dynamical system with a modified  flux which takes into account  the time dependency of the subspace. The resulting approximation (when discarding some necessary numerical approximations) interpolates the solution map $\Vxi \mapsto \Vu(\cdot,\Vxi)$ at points $\{\Vxi^1,\hdots,\Vxi^r\}$. An  a posteriori error estimate $\Delta_r(t,\Vxi)$ (local in time) is derived in the lines of 
\cite{Wirtz2014} using the logarithmic Lipschitz constant  associated to the flux. This error estimate is used in a  greedy procedure for the adaptive selection of interpolation points, where at step $r$, the next interpolation point corresponds to a maximizer over the parameters set $\VXi$ of a certain norm of $t \mapsto \Delta_r(t,\Vxi)$.\\

The paper is organized as follows. In Lemma \ref{sec:reductionmethod}, we introduce the Galerkin method for the projection of the dynamical system \eqref{eq:dynsystgen} onto a time-dependent reduced space and 
we derive an a posteriori error estimate.
In Lemma \ref{sec:lowrank}, we present strategies for the 
construction of reduced spaces $X_r$, including the classical POD-greedy strategy for the construction of time-independent reduced spaces, and the proposed greedy algorithm for the construction of  time-dependent reduced spaces.
In  Lemma \ref{sec:practical_aspects}, we detail some practical aspects for an implementation of the proposed method in a discrete time setting, and for obtaining  \emph{online}  computations (solution of the reduced dynamical system and evaluation of the error estimate) with a complexity independent of the dimension $d$ of the full-order model.  
In  Lemma \ref{sec:numres},  the proposed method is illustrated through numerical experiments on several test cases and compared to the POD-greedy approach. 

%%%%%%%%%%%%%%%%%%%%%%%%%%%%%%%%%%%%%%%%%%
%%%%%%%%%%%%%%%%%%%%%%%%%%%%%%%%%%%%%%%%%%
%%%%%%%%% Reduced dynamical system %%%%%%%%%%%%%%%%%
%%%%%%%%%%%%%%%%%%%%%%%%%%%%%%%%%%%%%%%%%%
%%%%%%%%%%%%%%%%%%%%%%%%%%%%%%%%%%%%%%%%%%

\section{Reduced dynamical system}\label{sec:reductionmethod}

In this section, we first propose a Galerkin method for computing a projection of the solution of the dynamical system \eqref{eq:dynsystgen}
onto a time-dependent subspace $X_r(t)$ of $X$. Then, we derive an a posteriori error estimate. Here, the reduced space 
$X_r(t)$ is supposed to be given.

\subsection{Projection}\label{subsec:proj}
The state space $X=\RR^d$ is equipped with the canonical inner product $\langle \Vx,\Vy \rangle_X := \Vx^T \Vy$ and associated norm  $\|\Vx\|_X$. 
We assume that 
\begin{align}
\dim X_r(t) =r \quad \text{for all } t>0.\label{ass:dim}
\end{align}
We denote by $\{\Vv_1(t),\hdots,\Vv_r(t)\}$ an orthonormal basis of $X_r(t)$ and we introduce the  orthogonal matrix 
 $$
\MV(t)=[\Vv_1(t),\dots,\Vv_r(t)]\in \RR^{d \times r},
$$
with $\MV^T(t)\MV(t) =\MI_r$, where $\MI_r$ is the identity matrix of size $r\times r$. We denote by 
$\MP_{X_r}(t) = \MV(t)\MV^T(t) \in \RR^{d\times d}$ the orthogonal projector onto $X_r(t)$.\\

 We define the approximation  $\Vu_r(t,\Vxi)$ of $\Vu(t,\Vxi)$ by projecting the equations of the dynamical system \eqref{eq:dynsystgen} onto $X_r(t)$
%\begin{align*}
%\MV(t)^T\Vu'_r(t,\Vxi) =  \MV(t)^T\Vf(\Vu_r(t),t,\Vxi), \quad t>0,%\label{eq:projdyn}
%\end{align*}
%or equivalently by
\begin{equation}
\left\{\begin{aligned}
\MP_{X_r}(t) \Vu'_r(t,\Vxi) &=  \MP_{X_r}(t)\Vf(\Vu_r(t,\Vxi),t,\Vxi),
\\
\Vu_r(0,\Vxi)&= \MP_{X_r}(0)\Vu_r^0(\Vxi).
\end{aligned}\right.
\label{eq:pertproj22}
\end{equation}
Note that for subspaces $X_r(t)$ generated by trajectories of the full-order model, $X_r(0)$ contains evaluations of the initial condition $\Vu_r^0(\Vxi)$, so that $\MP_{X_r}(0)\Vu_r^0 = \Vu_r^0$ in the case of a parameter-independent initial condition\footnote{Note that the dimension of $X_r(0)$ may be different from the dimension of  $X_r(t)$, $t>0$, which does not contradict the assumption \eqref{ass:dim}.}.
%\todo{with $\Vu_r^0(\Vxi)= \MP_{X_r}(0) \Vu^0(\Vxi)$ for parameter-dependent initial condition, or $\Vu_r^0=\Vu^0$ for parameter-independent initial condition.} 
We define
 $$\Vu_r(t,\Vxi) = \MV(t)\Valp(t,\Vxi),$$ 
with $\Valp(t,\Vxi)= \MV(t)^T\Vu_r(t,\Vxi)  \in \RR^r.$  
%We have
%$\MP'_{X_r} \Vu_r =(\MV{{\MV'}{}^{T}}+\MV'\MV^T)\MV\Valp= (\MI_d-  \MP_{X_r}) \MV' \Valp + \MV(\MV^T
%\MV)'\Valp =(\MI_d-  \MP_{X_r}) \MV'\Valp$. 
Let $ \MP_{X_r^\perp}(t) = \MI_d-  \MP_{X_r}(t)$ be the orthogonal projector onto $X_r^\perp(t)$ which is the orthogonal complement of $X_r(t)$ in $X$. 
From the first equation of \eqref{eq:pertproj22} and $\MP_{X_r^\perp}(t)\MV(t) = 0$, we deduce 
\begin{equation}
\Vu'_r(t,\Vxi)  =\MP_{X_r}(t)\Vf(\Vu_r(t,\Vxi) ,t,\Vxi)+\MP_{X_r^\perp}(t) \MV ' (t)\Valp(t,\Vxi).
\label{eq:pertproj3}
\end{equation}
A reduced dynamical system of dimension $r$ is then obtained for the components $\Valp$:
\begin{equation}\left\{
\begin{aligned}
\Valp'(t,\Vxi) &=  \Vf_r(\Valp(t,\Vxi),t,\Vxi),\\
\Valp(0,\Vxi) &= \Valp^0(\Vxi),
\end{aligned}\right.
\label{eq:pertproj2}
\end{equation} 
with $\Valp^0(\Vxi) = \MV^T(0)\Vu^0(\Vxi)$ and a reduced flux $\Vf_r$ defined by 
\begin{equation}
\label{eq:reddyn}
 \Vf_r(\Valp,t,\Vxi) = \MV(t)^T \Vf(\MV(t)\Valp ,t,\Vxi)-  {\MV(t)^T\MV'(t)\Valp},
\end{equation} 
where the last term takes into account the time dependency of the reduced basis. For a time-independent reduced basis, i.e. such that $\MV'(t)=0$, we recover a classical projected dynamical system (see, e.g., \cite{Haasdonk2011}).

\begin{remark}
Assuming that for a fixed $\Vxi$, $\Vv \mapsto \Vf(\Vv,t,\Vxi)$ is uniformly Lipschitz continuous, $(\Vv,t) \mapsto \Vf(\Vv,t,\Vxi)$ is continuous, and assuming that 
$\MV$ is continuously differentiable with $\MV'$ uniformly bounded, then $\Vv \mapsto \Vf_r(\Vv,t,\Vxi)$ is uniformly Lipschitz continuous with respect to $\Vv$ and $(\Vv,t) \mapsto \Vf_r(\Vv,t,\Vxi)$ is continuous. Then, the reduced dynamical system \eqref{eq:pertproj2} admits a unique solution $t \mapsto\Valp(t,\Vxi) $ which is continuously differentiable.
\end{remark} 

\subsection{A posteriori error estimate}\label{subsec:error}

From equations \eqref{eq:dynsystgen} and \eqref{eq:pertproj3}, 
we deduce that the error $\Ve_r= \Vu- \Vu_r$ satisfies 
\begin{equation}
\begin{array}{rcl}
 \Ve_r'(t,\Vxi)
&=&
 \underbrace{\Vf(\Vu(t,\Vxi),t,\Vxi) - \Vf(\Vu_r(t,\Vxi),t,\Vxi)}_{{  \boldsymbol{\delta}_1 }} 
 +\underbrace{\MP_{X_r^\perp}(t)  \Vf(\Vu_r(t,\Vxi),t,\Vxi)}_{{  \boldsymbol{\delta}_2 }},\\
 &-& \underbrace{\MP_{X_r^\perp}(t) \MV'(t) \MV(t)^T\Vu_r(t,\Vxi) }_{{  \boldsymbol{\delta}_3 }}  ,
 \end{array}
 \label{eq:SDE}
\end{equation}
with   $\Ve_r(0,\Vxi) =\MP_{X_r^\perp} (0) \Vu^0(\Vxi)$.
 The time derivative of the error $\Ve_r$ is the sum of three contributions: the error $ \boldsymbol{\delta}_1$ between the flux  
 and its approximation, the error $ \boldsymbol{\delta}_2$ between the flux approximation and its projection onto  $X_r(t)$, 
 and an additional term $ \boldsymbol{\delta}_3$ taking into account the time dependency of the basis.\\
 
We first recall the definition of the local logarithmic Lipchitz constant, as defined in \cite[\S 2.1]{Wirtz2014}. 
This constant will provide a local information on the flux $\Vf$ around the approximation $\Vu_r$ and 
 will allow to derive an a posteriori error estimate. 
\begin{definition}[Local logarithmic Lipschitz constant]\label{def:lipschitz}
 For a Lipschitz continuous function $\Vf : X \to X$,
 the local logarithmic Lipschitz constant of $\Vf$ at $\Vv \in X$ is defined by 
$$
L_X[\Vf](\Vv) = \sup_{ \Vu\in X, \Vu\neq\Vv } \dfrac{\langle \Vu-\Vv, \Vf(\Vu)-\Vf(\Vv)\rangle_X}
{\|\Vu-\Vv\|^2_X}.
$$
\end{definition}
%\todo{We first recall the definition of the logarithmic Lipchitz constant for a Lipschitz continuous function. 
%\begin{definition}[Logarithmic Lipschitz constant]  \label{def:lipschitz}
%For a  function $\Vf : \RR^d \to \RR^d$, the logarithmic Lipschitz constant with
%respect to $\|\cdot \|_X$ is defined by 
%$$
%L_X[\Vf] =  \lim_{h \to 0^+} \frac{1}{h}
%\left(
%\sup_{ \Vu,\Vv \in \RR^d,\Vu \neq\Vv  } \dfrac{\|\Vu-\Vv+ h (\Vf(\Vu)-\Vf(\Vv))\|_X}{\|\Vu-\Vv\|_X} - 1
%\right)
%$$
%\end{definition}
%
%The following alternative definition of logarithmic Lipschitz constant   is considered in practice (see \cite[\S 2.1]{Wirtz2014}). We also give a definition of  the {\em local} logarithmic Lipschitz constant, since local information on the flux around the current reduced state  is required in the proposed error estimate.
%\begin{lemma}
%Given a Lipschitz continuous function $\Vf : \RR^d \to \RR^d$, i.e. there exists a constant $K_X[\Vf]>0$ such that $\|\Vf(\Vu)-\Vf(\Vv)\|_X \le K_X[\Vf] \|\Vu - \Vv\|_X$, the logarithmic Lipschitz constant of $\Vf$ with respect to $\|\cdot\|_X$ reads
%$$
%L_X[\Vf] = \sup_{ \Vu,\Vv \in \RR^d,\Vu \neq\Vv} \dfrac{\langle \Vu-\Vv, \Vf(\Vu)-\Vf(\Vv)\rangle_X}
%{\|\Vu-\Vv\|^2_X}.
%$$
%and the local logarithmic Lipschitz  for $\Vv \in X$ is given by 
%$$
%L_X[\Vf](\Vv) = \sup_{ \Vu,\Vv\in \RR^d, \Vu\neq\Vv } \dfrac{\langle \Vu-\Vv, \Vf(\Vu)-\Vf(\Vv)\rangle_X}
%{\|\Vu-\Vv\|^2_X}.
%$$
%\end{lemma}}
For an affine flux, the local logarithmic Lipschitz constant is constant and can be obtained by solving an eigenvalue problem.
Indeed, if $\Vf(\Vu) = \MA \Vu + \Vg$ with $\MA \in  \RR^{d \times d}$ and $\Vg \in \RR^d$, 
then 
\begin{equation}\label{eq:liplin}
L_X[\Vf](\Vv) = L_X[\MA]= \sup_{0\neq \Vu \in X } \dfrac{\langle \Vu , \MA \Vu\rangle_X}{\|\Vu\|^2_X}
= \lambda_{max}\left(\frac{\MA+\MA^T}{2}\right),
\end{equation}
where $\lambda_{max}(\MM)$ denotes the maximum eigenvalue of the matrix $\MM \in \RR^{d \times d }$.\\

Now, we recall a comparison lemma \cite{Wirtz2014} in a form suitable for the derivation of our a posteriori error estimate. 
\begin{lemma}[Comparison lemma]\label{lem:comparison}
 Let $T>0$ and let $u,\alpha,\beta : [0,T] \to \RR$ be integrable functions. Assume that $u$ is differentiable
 and  $ u' \le \beta u  + \alpha$. Then for all $ t \in [0,T]$, it holds $u(t) \le v(t)$, where $v(t)$ is the solution of the differential equation $v' = \beta v + \alpha$ with initial  condition $v(0)=u(0)$.
%  $$
% u(t) \le \int_0^t \alpha(s) e^{\int_s^t \beta(\tau) d\tau}ds + u(0)e^{\int_0^t \beta(\tau) d\tau}.
% $$
\end{lemma}
\begin{proof} 
Denoting $\gamma(t) = \int_0^t \beta(\tau) d\tau$, we have
$$
\begin{array}{rcl}
\displaystyle u(t) &=& e^{\gamma(t)}e^{-\gamma(t)}u(t) -e^{\gamma(t)}u(0) +e^{\gamma(t)}u(0),\\
&=& e^{\gamma(t)}\int_0^t (e^{-\gamma(\tau)}u(\tau))' d\tau +e^{\gamma(t)}v(0),\\
&\le & e^{\gamma(t)} \left( \int_0^t e^{-\gamma(\tau)}\alpha(\tau) d\tau +v(0)\right) = v(t),
\end{array}
$$ 
which ends the proof.
\end{proof}
Following \cite{Wirtz2014}, we now provide a bound for the error norm $\norm{\Ve_r(t,\Vxi)}_X$. 
%We omit the dependence on $\Vxi$  for the sake of clarity. 
\begin{proposition}\label{thm:complex}
The error norm $\|\Ve_r(t,\Vxi)\|_X$ satisfies 
\begin{equation}
\|\Ve_r(t,\Vxi)\|_X \le \Delta_r(t,\Vxi) \label{bound_DeltaNL}
\end{equation}
for all $t\ge 0$,
where $\Delta_r(t,\Vxi)$ is the solution of the ordinary differential equation
\begin{equation}
\label{eq:errorEDO}
\left\{\begin{aligned}
\Delta_r'(t,\Vxi) &=  L_X[\Vf](\Vu_r(t,\Vxi)) \Delta_r(t,\Vxi) + \|\Vr(t,\Vxi) \|_X, \\
\Delta_r(0,\Vxi) & = \| \Ve_r(0,\Vxi)\|_X,
\end{aligned}
\right.
\end{equation}
with $\Vr(t,\Vxi) = \MP_{X_r^\perp}(t)(\MV'(t) \MV(t)^T\Vu_r(t,\Vxi)- \Vf(\Vu_r(t,\Vxi),t,\Vxi)  ) $ and $\Ve_r(0,\Vxi) =\Vu^0(\Vxi) -\MP_{X_r}(0)\Vu^0(\Vxi) $.
%\begin{align*}
%&\beta(t) = L_X[\Vf](\Vu_r(t)),\\
%&\gamma(t) = \norm{\MQ_r(t)( \Vf(\Vu_r(t),t) -\MV'(t) \Valp(t) ) }_X,\\
%& \| \Ve(0)\|_X = \|\Vu^0 -\MP_{X_r}(0)\Vu^0\|_X.
%\end{align*}
%\begin{equation}
%\Delta(t) =  \int_0^t \gamma(s) e^{\int_s^t \beta(\tau) d\tau ds} + \norm{\Ve(0)}_X e^{\int_0^t \beta(\tau) d\tau ds},
%\end{equation}
%where 
\end{proposition}
\begin{proof} For the sake of clarity, we omit the dependence on $\Vxi$ in the proof.
Taking the scalar product of \eqref{eq:SDE} with $\Ve_r(t)$, it comes
\begin{align*}
 \frac{1}{2} \frac{d}{dt}\|\Ve_r(t)\|^2_X &= \langle\Ve_r(t),\Ve_r'(t)\rangle_X,\\
&= \langle \Ve_r(t),(\Vf(\Vu(t),t)-\Vf(\Vu_r(t),t))+\Vr(t) \rangle_X,\\ 
&\le L_X[\Vf](\Vu_r(t))\|\Ve_r(t)\|^2_X + \|\Ve_r(t)\|_X\|\Vr(t)\|_X,
\end{align*}
where we have used the definition \ref{def:lipschitz} of the local logarithmic Lipschitz constant at $\Vu_r(t)$. We then obtain  
$ \frac{d}{dt}\|\Ve_r(t)\|_X \le L_X[\Vf](\Vu_r(t))\|\Ve_r(t)\|_X + \|\Vr(t)\|_X$, and the result follows from   Lemma \ref{lem:comparison}.
\end{proof}
In practice, the error bound $\Delta_r(t)$ will be estimated by solving approximately 
the differential  equation \eqref{eq:errorEDO} using a numerical scheme (see Lemma \ref{sec:practical_aspects}).

\begin{remark}
Following the proof of Theorem \ref{thm:complex}, we easily prove that the solution of \eqref{eq:errorEDO}, with $L_X[\Vf](\Vu_r)$ replaced by the local Lipschitz constant 
$$
K_X[\Vf](\Vu_r) = \sup_{\Vv \in X,\Vv\neq \Vu_r} \frac{\Vert \Vf(\Vv) - \Vf(\Vu_r) \Vert_X}{\Vert \Vv - \Vu_r \Vert_X},
$$ 
also provides an error bound. However, since $L_X[\Vf]\le K_X[\Vf]$, a sharper error bound is obtained when using 
 the local logarithmic Lipschitz constant. Moreover, $L_X[\Vf](\Vu_r)$ is easier to compute than $K_X[\Vf](\Vu_r)$ in practice.
\end{remark}

{  
\begin{remark}
The error bound $\Delta_r(t)$ may grow exponentially with time. However, it is a certified error bound and a good candidate for piloting adaptive algorithms (see later). Moreover, if particular information is available on the structure of the problem (e.g. if it arises from spatial discretization of a PDE), it is possible to improve this error bound using energy norms \cite[\S 3.3]{Haasdonk2014}. In the context of Petrov-Galerkin space-time formulations, better error estimates can be obtained for particular classes of problems, e.g. \cite{Urban2012} for parabolic equation and \cite{Yano2014} for Burgers' equation. For the general non linear dynamical systems considered here, a way to improve the effectivity of  error bound $\Delta_r$ is to improve the local Lipschitz constant, as suggested in \cite[\S 5]{Wirtz2014}.
\end{remark}
 }

%%%%%%%%%%%%%%%%%%%%%%%%%%%%%%%%%%%%%%%%%%
%%%%%%%%%%%%%%%%%%%%%%%%%%%%%%%%%%%%%%%%%%
%%%%%%%%% Greedy algorithms %%%%%%%%%%%%%%%%%
%%%%%%%%%%%%%%%%%%%%%%%%%%%%%%%%%%%%%%%%%%
%%%%%%%%%%%%%%%%%%%%%%%%%%%%%%%%%%%%%%%%%%

\section{Greedy algorithms for the construction of reduced spaces}\label{sec:lowrank}

In this section, we introduce greedy algorithms for the construction of an increasing sequence of reduced spaces $\{X_r\}_{r>0}$. These spaces are generated from successive evaluations of the solution $\Vu$ of the full-order model at parameters values $\{\Vxi^r\}_{r>0}$ which are selected adaptively.  

%First, in Lemma \ref{sec:timeindep}, we consider the construction of time-independent reduced spaces with a particular emphasis on the so-called POD-greedy algorithm. Then, in Lemma \ref{sec:timedep}, we introduce strategies for constructing time-dependent spaces.
%
%\subsection{Preliminaries}
For a given subspace $X_r$ in $X$, possibly time-dependent, we consider $\Vu_r(\cdot,\Vxi)$ 
the solution of  \eqref{eq:pertproj22}.
We assume that an a posteriori estimate 
$\Delta_r(t,\Vxi)$ of the error $\| \Ve_r(t,\Vxi)\|_X$ is available (see Subsection \ref{subsec:error}).
Defining 
$$
X_{r+1}(t) =  X_r(t) + \mathrm{span} \{\Vu(t,\Vxi^{r+1}(t))\},
$$
 where $\Vxi^{r+1}(t)$ maximizes $\Vxi \mapsto \Delta_r(t,\Vxi)$ over $\VXi$, would result in a standard greedy algorithm for the approximation of the solution manifold $\{\Vu(t,\Vxi):\Vxi\in \VXi\}$ at time $t$ with a sequence of subspaces $\{X_r(t)\}_{r>0}$ (see \cite{Buffa2012,Binev2011,DeVore:2013fk} for convergence results on greedy algorithms). However, choosing time-dependent parameters values is infeasible in practice, even when working with a discretization on a time grid. Indeed, it requires the solution of the full-order model for too many parameters values.

Therefore, we rely on evaluations of the solution at time-independent parameters values whose selection requires a global in time error estimate. 
Here, we introduce an  a posteriori estimate of the $L^2(0,T)$-norm of the error $\| \Ve_r(\cdot,\Vxi) \|_X$ defined by
$$
\Delta_r^{(0,T)}(\Vxi) := \Vert \Delta_r(\cdot, \Vxi) \Vert_{(0,T),2} = \left(\int_{0}^T \Delta_r(t, \Vxi)^2dt\right)^{1/2},
$$ 
where $\Vert \cdot \Vert_{(0,T),2}$ denotes the natural norm in $L^2(0,T)$.
 \begin{remark}
Note that other norms of $t \mapsto \Delta_r(t, \Vxi)$ (e.g., the $L^\infty(0,T)$-norm or a weighted $L^2(0,T)$-norm) could be used for defining the error indicator 
 $\Delta_r^{(0,T)}(\Vxi)$, depending on how we want to control the quality of the reduced-order model. 
 \end{remark}

\subsection{T-greedy algorithm} \label{sec:timedep}

We first propose a very natural strategy which consists in defining  
$$
X_{r+1}(t) =  X_r(t) + \mathrm{span} \{\Vu(t,\Vxi^{r+1})\}. 
$$
A basic strategy would consist in choosing the sequence of parameters values 
$\{\Vxi^1,\dots,\Vxi^r,\hdots\}$ at random. 
 In this work, we adopt an adaptive greedy strategy  called  {\it T-greedy} algorithm, 
where $\Vxi^{r+1}$ maximizes over $\VXi$ the error estimate $\Delta_r^{(0,T)}(\Vxi)$.  
Note that in practice, parameters values are selected  in a finite subset $\VXi_{train}$ in $\VXi$ called a \emph{training set}. 
The T-Greedy algorithm is summarized in Algorithm \ref{alg:tgreedy}.

\begin{algorithm}
\caption{T-Greedy algorithm} \label{alg:tgreedy}
\begin{algorithmic}[1]
\STATE Set $r=0$, $X_0=\{0\}$.
  \STATE Compute $\Vu_{r}(\cdot,\Vxi)$ and $\Delta_{r}(\cdot,\Vxi)$ for all $\Vxi$ in $\VXi_{train}$.\label{computeurTgreedy}
  \STATE Find $ \displaystyle \Vxi^{r+1} \in \arg\max_{\Vxi \in \VXi_{train}}\Delta_r^{(0,T)}(\Vxi) $.\label{choosepoint}
  \STATE  Compute $t \mapsto \Vu(t,\Vxi^{r+1})$. 
  \STATE Set $X_{r+1}(t) = X_{r}(t) + \spann\{\Vu(t,\Vxi^{r+1})\}$.
  \STATE Set $r \leftarrow r+1$ and go to step \ref{computeurTgreedy}.
\end{algorithmic}
\end{algorithm}
In practice, we can fix a desired precision $\varepsilon$ and  stop the 
Algorithm \ref{alg:tgreedy} after step \ref {computeurTgreedy} if $\max_{\Vxi \in \VXi_{train}}\Delta_r^{(0,T)}(\Vxi) < \varepsilon$.
At iteration $r$ of the algorithm, 
the approximation $\Vu_r(t)$ takes the form 
\begin{align}
\Vu_r(t,\Vxi) =  \sum_{i=1}^r \Vv_i(t) \alpha_i(t,\Vxi) ,
\end{align}   
where $\{\Vv_1(t),\hdots,\Vv_r(t)\}$ constitutes a basis of the space $X_r(t)$. 
The time-dependent subspace $X_r(t)$ contains the solution $\Vu(t,\Vxi)$ of the full-order model for parameters values $\Vxi^1,\hdots,\Vxi^r$. Therefore, by the property of the Galerkin projection, the solution $\Vu_r$ of the reduced dynamical system interpolates the solution $\Vu$ at these parameters values.  \\

Satisfying assumption \eqref{ass:dim} requires that for all $t > 0$, the vectors $\Vu(t,\Vxi^{1}),\hdots,\Vu(t,\Vxi^r)$ are linearly independent. Under this assumption, we can define the basis of 
 $X_r(t)$ by  
\begin{equation}\label{eq:v}
\Vv_i(t) = \Vu(t,\Vxi^{i}) \quad \text{ or } \quad \Vv_{i}(t) =\dfrac{ (\MI_d-\MP_{X_{i-1}}(t))\Vu(t,\Vxi^{i})}{\norm{ (\MI_d-\MP_{X_{i-1}}(t))\Vu(t,\Vxi^{i})}_X},
\end{equation}
the second choice resulting in an orthonormal basis which is more convenient for numerical stability issues. 
Moreover, assuming that the flux $\Vf(\cdot,\cdot,\Vxi^i)$ is  continuous and uniformly Lipschitz continuous 
with respect to its first variable, the trajectories $t\mapsto \Vu(t,\Vxi^i)$, and therefore the functions $t\mapsto \Vv_i(t)$, are continuously differentiable.

\subsection{POD-Greedy algorithm} \label{sec:timeindep}

Here, we describe the POD-greedy algorithm introduced in  \cite{Haasdonk2008} for the adaptive construction of time-independent reduced spaces $X_r$. Given an approximation space $X_r$ and the corresponding solution $\Vu_r$ of the reduced dynamical system in $X_r$, we select a new parameters value
$\Vxi^{r+1}$ which maximizes   the 
 error estimate $\Vxi \mapsto \Delta_r^{(0,T)}(\Vxi)$, such as for the T-greedy strategy.
Then the trajectory $\Vu(\cdot,\Vxi^{r+1})$ of the full-order model is computed and
the space $X_r$ is enriched by the $\ell$-dimensional subspace generated by the first $\ell$ POD modes $POD_{\ell}(\Vs_r(\cdot,\Vxi^{r+1}))$ of $\Vs_r(\cdot,\Vxi^{r+1}) = \Vu(\cdot,\Vxi^{r+1}) - \MP_{X_r} \Vu(\cdot,\Vxi^{r+1})$, such that (see, e.g., \cite{Volkwein2008})
$$
\mathrm{span} (POD_{\ell}(\Vs_r(\cdot,\Vxi^{r+1}))) = \arg\min_{\dim(V)=\ell}
\int_{0}^T  \| \Vs_r(t,\Vxi^{r+1})- \MP_{V} \Vs_r(t,\Vxi^{r+1})\|_X^2 dt.
$$ 
The algorithm is summarized in Algorithm \ref{alg:PODG}.

\begin{algorithm}\caption{POD-greedy algorithm \label{alg:PODG}}
\begin{algorithmic}[1]
%\REQUIRE $\varepsilon_{tol}>0$ and $r_{max} \in \NN^*$.
\STATE Set $r=0$, $X_0=\{0\}$.
%\WHILE{$\displaystyle \max_{\Vxi \in \VXi_{train}} \Delta_r^{(0,T)}(\Vxi) > \varepsilon_{tol}$
% and $r<r_{max}$}
%\WHILE{$r<r_{max}$}
\STATE Compute $\Vu_r(\cdot,\Vxi)$ and $\Delta_{r}(\cdot,\Vxi)$ for all $\Vxi$ in $\VXi_{train}$.\label{computeurPODgreedy}
\STATE Find $  \Vxi^{r+1} \in \arg\max_{\Vxi \in \VXi_{train}} \Delta_r^{(0,T)}(\Vxi)$. \label{chooseparamPOD}
\STATE Compute $\Vu(\cdot,\Vxi^{r+1})$ and $\Vs_r(\cdot, \Vxi^{r+1}) = \Vu(\cdot,\Vxi^{r+1})-\MP_{X_r}\Vu(\cdot,\Vxi^{r+1})$.
\STATE Set  $X_{r+\ell} =  X_{r} + \mathrm{span}(POD_{\ell}\left(\Vs_r(\cdot, \Vxi^{r+1})\right))$.
\STATE Set $r \leftarrow r + \ell$ and go to step \ref{computeurPODgreedy}.
%\ENDWHILE
\end{algorithmic}
\end{algorithm}
As for Algorithm \ref{alg:tgreedy}, we can fix a desired precision $\varepsilon$ and  stop the Algorithm \ref{alg:PODG}
 after step \ref{computeurPODgreedy} if $\arg\max_{\Vxi \in \VXi_{train}}\Delta_r^{(0,T)}(\Vxi) < \varepsilon$.
%\begin{remark}
%If the error estimate is such that 
%%
%%Efficient error estimator $\Delta_r$ have been derived for linear dynamical system \cite{Haasdonk2011} and then for nonlinear systems \cite{Wirtz2014}.  Moreover, if  
%\begin{equation}
%c_r \Delta_r(t,\cdot) \le \| \Vu(t,\cdot) - \Vu_r(t,\cdot)\|_X \le C_r \Delta_r(t,\cdot),% \quad t>0,
%\label{eq:surrogate}
%\end{equation}
%with $0<c_r\le C_r<\infty$ some positive constants, the POD-greedy algorithm has been proved to converge with quasi-optimal rates \cite{Haasdonk2013}.
%\end{remark}
Let us remark that choosing $\ell=1$ allows to obtain a slow increase of the dimension of $X_r$ along the iterations.  
Since the reduced space is enriched by only the first modes of 
the POD of the trajectories $t \mapsto \Vu(t,\Vxi^i)$, $1\le i \le r$, the approximation $\Vu_r$ does not in general  interpolate the function $\Vxi \mapsto \Vu(\cdot,\Vxi)$ at points  $\Vxi^1,\hdots,\Vxi^r$. As it will be illustrated in the numerical experiments, this enrichment strategy may yield very high-dimensional reduced spaces for reaching a desired accuracy. 
%Moreover, POD-greedy procedure could generate high dimensional basis when
%the spectral content of the solution  trajectory is rich. 
A typical example is the advection problem  $\partial_t u + \xi \partial_x u =0$ on the torus with initial condition $u^0$, for which a very high-dimensional time-independent reduced space may be required to approximate the solution
$t \mapsto u^0(x-\xi t)$, even for one instance of the parameters $\Vxi$. More generally, it is well known that  POD is not well suited for solving problems with propagating fronts.

%%%%%%%%%%%%%%%%%%%%%%%%%%%%%%%%%%%%%%%%%%
%%%%%%%%%%%%%%%%%%%%%%%%%%%%%%%%%%%%%%%%%%
%%%%%%%%% Practical aspects %%%%%%%%%%%%%%%%%
%%%%%%%%%%%%%%%%%%%%%%%%%%%%%%%%%%%%%%%%%%
%%%%%%%%%%%%%%%%%%%%%%%%%%%%%%%%%%%%%%%%%%

\section{Practical implementation}\label{sec:practical_aspects}

In this section, we provide the practical ingredients for an efficient offline/online implementation of the proposed model reduction method in a discrete time setting. We first introduce a time integration scheme 
for the solution of the full and reduced order dynamical systems, and for the computation of the error estimate. Then, for the model order reduction method to be efficient, the computation of the approximation $\Vu_r$ as well as the evaluation of the error estimate $\Delta_r$ have to be performed online with a complexity independent of the dimension $d$ of the full-order model. This requires some assumptions on the dependence of the flux on the parameters and some pre-computations in an offline phase.

%For linear dynamical systems, with a flux $\Vf()$ 
% an offline/online strategy. The offline stage 
%For model reduction to make sense 
%The solution of the reduced system \eqref{eq:discproj2} should br as well as the computation of the error estimators provided by Theorem \ref{thm:complex} must be done efficiently with a complexity depending only 
%on $r$. When dealing with linear dynamical systems, in particular with fluxes that admits affine-linear decomposition 
%\eqref{eq:afflin}, it is possible to proceed in {\it online/offline} fashion as usually done in the context of RB. 

%
%
%When solving nonlinear dynamical system, we have to address the task of computing efficiently the nonlinear flux  $\Vf$
%and the logarithmic Lipschitz constants. To this goal, we introduce additional approximation 
%of $\Vh$ and $L_X[\Vf]$ using interpolation methods.
%

\subsection{Time integration}

Let $\mathbb{T}=\{t_k\}_{k=0}^K$ be a regular discretization of $[0,T]$ with $t_k = k\delta t$ and $\delta t = \frac{T}{K}$. Given $a : [0,T] \to V$, with $V$ a vector space, we denote by 
$a^k\approx a(t^k)$ an approximation of $a$ at time $t^{k}$ and $\delta a^k = a^{k+1} - a^{k}$.
We assume that the flux $\Vf : X\times (0,T) \times \VXi  \to X$ can be decomposed as follows:
$$
 \Vf(\Vu,t,\Vxi)= \MA(t,\Vxi) \Vu + \Vh(\Vu,t,\Vxi) + \Vg(t,\Vxi),
$$
where $ \MA(t,\Vxi)  \in \RR^{d \times d}$, $\Vh(\cdot,t,\Vxi) :X \to \RR^d$ and $\Vg(t,\Vxi)\in \RR^d$. 
For solving \eqref{eq:dynsystgen}, 
we use the following semi-implicit time integration scheme  
\begin{equation}
\label{eq:discretdyn}
(\MI_d- \delta t\MA^{k+1}(\Vxi))\Vu^{k+1}(\Vxi) = \Vu^k(\Vxi)+ \delta t \Vh^k(\Vu^k(\Vxi),\Vxi) + \delta t \Vg^k(\Vxi),\\
\end{equation}
%where 
%$$
% \Vf_E^k(\Vu^k(\Vxi),\Vxi)=\Vh^k(\Vu^k(\Vxi),\Vxi)+  \Vg^k(\Vxi).
%$$
%The time step $\delta t$ is chosen in order to satisfy the Courant Friedrichs Levy (CFL) condition, which  ensures  numerical stability.
The matrix  $(\MI_d- \delta t\MA^{k+1}(\Vxi))$ is assumed to be invertible.  

\subsubsection{Time integration of the reduced dynamical system}

Let $X_r^k$ denote the reduced space at time $t^k$, $\{\Vv_1^k,\dots,\Vv_r^k\}$ an orthonormal basis of $X_r^k$, $\MV^k = [\Vv_1^k,\dots,\Vv_r^k]$, $\MP_{X_r^k} = {\MV^k}^T\MV^k$ the orthogonal projector  onto $X_r^k$, and $ \MP_{{X_r^k}^\perp}=\MI_d- \MP_{X_r^k}$. The  approximation $\Vu_r^k$ at time $t^k$  is obtained by projecting the discrete dynamical system \eqref{eq:discretdyn} onto $X_r^{k+1}$ 
\begin{equation}
\label{eq:discproj}
\Vu_r^{k+1}(\Vxi)  = \MP_{X_r^{k+1}} \left( \Vu_r^{k}(\Vxi) + \delta t (\MA^{k+1}(\Vxi) \Vu_r^{k+1}(\Vxi) +  \Vh^k(\Vu_r^k(\Vxi),\Vxi) + \Vg^k(\Vxi))\right), 
\end{equation}
with 
$\Vu_r^0(\Vxi) = \MP_{X_r^{0}} \Vu^0(\Vxi)$. 
%Noting that  $\MP_{{X_r^{k+1}}^\perp}\Vu_r^k(\Vxi)= -\MP_{{X_r^{k+1}}^\perp}\delta \MV^k \Valp^k(\Vxi)$, 
% \eqref{eq:discproj} can be rewritten 
%\begin{align*}
%& (\MI_d- \delta t\MA^{k+1}(\Vxi))\Vu_r^{k+1}(\Vxi)= \Vu_r^k(\Vxi) + \delta t  \Vh^k(\Vu_r^k(\Vxi),\Vxi) + \delta t \Vg^k(\Vxi)
%\\
%\quad &- \delta t \MP_{{X_r^{k+1}}^\perp} \left(\MA^{k+1}(\Vxi) \Vu_r^{k+1}(\Vxi) + \Vh^k(\Vu_r^k(\Vxi),\Vxi) + \Vg^k(\Vxi)-\frac{\delta \MV^k}{\delta t } \Valp^k(\Vxi) \right).
%\end{align*}
%Noting that ${\MV^{k+1}}^T \Vu_r^k =  \Valp^{k}(\Vxi) - \delta \MV^k \Valp^{k}(\Vxi)$, 
Then, we obtain the following discrete reduced dynamical system:
\begin{align}
\label{eq:discretereduced}
\left(\MI_r - \delta t \MA_{r}^{k+1}(\Vxi)\right)\Valp^{k+1}(\Vxi) &= \Valp^{k}(\Vxi) +\\
\delta t {\MV^{k+1}}^T &\left(   \Vh^{k}(\MV^k\Valp^k(\Vxi),\Vxi) +  \Vg^k(\Vxi) - \frac{\delta \MV^{k}}{\delta t} \Valp^k(\Vxi)\right),  \nonumber
% {\MV^{k+1}}^T \left(\MV^{k} \Valp^k(\Vxi) + \delta t \Vh^{k}(\MV^k\Valp^k(\Vxi),\Vxi) + \delta t \Vg^k(\Vxi)\right),  
\end{align}
with initial condition $\Valp^0(\Vxi) = {\MV^0}^T\Vu^0(\Vxi)$, where
\begin{equation}
\label{eq:redA}
 \MA_{r}^{k+1}(\Vxi)  = {\MV^{k+1}}^T \MA^{k+1}(\Vxi)\MV^{k+1},
\end{equation}
and where $\frac{\delta \MV^{k}}{\delta t}$ takes into account the possible time dependency of the reduced basis (this term is equal to zero for time-independent reduced spaces). 
% and the reduced flux
% \begin{equation}
%\label{eq:redF}
% \Vf_{E,r}^{k+1,k}(\Vxi)= {\MV^{k+1}}^T\Vf_{E}^{k}(\Vxi) = {\MV^{k+1}}^T\Vh^k(\Vu^k(\Vxi),\Vxi)+  {\MV^{k+1}}^T\Vg^k(\Vxi).
% \end{equation}
In practice, for an efficient solution of \eqref{eq:discretereduced}, $\Vh^{k}(\Vu^k(\Vxi),\Vxi)$ is replaced by an approximation $\tilde\Vh^{k}(\Vu^k(\Vxi),\Vxi)$ using an empirical interpolation method (see  Subsection \ref{subsub:reducflux}).  

\subsubsection{Time integration for error estimation}

We first note that if in \eqref{eq:errorEDO} the constant $L_X[\Vf](\Vu_r)$ is replaced by an upper bound, then  the solution of the ordinary differential equation provides an upper bound for $\Delta_r$.  Noting that $L_X[\Vf](\Vu_r) \le L_X[\MA] + L_X[\Vh](\Vu_r)$, we introduce estimations $\tilde L_X[\MA] $ and $\tilde L_X[\Vh](\Vu_r)$ of $ L_X[\MA] $ and $ L_X[\Vh](\Vu_r)$ respectively and consider the following ordinary differential equation whose solution $\tilde \Delta_r$ provides an estimation of an upper bound of $\Delta_r$:
\begin{equation}
\label{eq:errorEDO_bound}
\left\{\begin{aligned}
\tilde \Delta_r'(t,\Vxi) &= \tilde L_X[\MA]\tilde \Delta_r(t,\Vxi) + \tilde L_X[\Vh](\Vu_r)\tilde \Delta_r(t,\Vxi) + \|\tilde \Vr(t,\Vxi) \|_X, \\
\tilde \Delta_r(0,\Vxi) & = \| \Ve_r(0,\Vxi)\|_X.
\end{aligned}
\right.
\end{equation}
Here,  $\|\tilde \Vr(t,\Vxi) \|_X$ denotes an approximation of $\|\Vr(t,\Vxi) \|_X$ where $\Vh(\Vu_r(t,\Vxi),t,\Vxi)$ is replaced by an approximation 
$\tilde \Vh(\Vu_r(t,\Vxi),t,\Vxi)$  obtained by an empirical interpolation  method (see Subsection \ref{subsub:reducflux}).

Equation \eqref{eq:errorEDO_bound} is then solved using the following semi-implicit scheme 
(consistent with \eqref{eq:discretdyn}):
\begin{align}
\tilde \Delta^{k+1}_r(\Vxi)  =  (1- \delta t \tilde L_X[\MA^{k+1}(\Vxi)])^{-1} \Big( \tilde \Delta^k_r(\Vxi) + 
\delta t \tilde L_X[\Vh^k](\Vu_r^k(\Vxi))  + \delta t \Vert{\tilde \Vr^{k}(\Vxi)} \Vert_X \Big).  \label{eq:estimate2} 
\end{align}
with
\begin{equation}
\tilde \Vr^{k}(\Vxi) = \MP_{{X_r^{k+1}}^\perp}\left(\frac{\delta\MV^k}{\delta t} \Valp^k(\Vxi)- \MA^{k+1}(\Vxi) \Vu_r^{k+1}-\tilde\Vh^k(\Vu_r^k(\Vxi),\Vxi) - \Vg^k(\Vxi)  \right) ,
\label{tilder_discrete}
\end{equation}
where $\frac{\delta \MV^{k}}{\delta t}$ takes into account the possible time dependency of the reduced basis (this term is equal to zero for time-independent reduced spaces).

\subsection{Offline/Online implementation}

A time and parameter-dependent function $a(t,\Vxi)$ with values in some vector space $V$ is said to admit a time-dependent affine representation if 
$$
a(t,\Vxi) = \sum_{i=1}^{Q_{\MA}} \theta^i_a(t,\Vxi) a^i(t),
$$
with $\theta^a_i(t,\Vxi) \in \RR$ and $a^i(t)\in V$. For a function discretized on a time grid $\{t^k\}_{k=0}^K$, that means 
 $
a^k(\Vxi) = \sum_{i=1}^{Q_{\MA}} \theta^{i,k}_a(\Vxi) a^{i,k},
$ for $0 \le k \le K$.
For the sake of presentation, we keep the notation $a(t,\Vxi)$ for time-dependent functions, even if $t$ only takes a finite set of values $t^k$, $0\le k \le K$. 
Then, we assume that 
$\MA(t,\Vxi)$ and $\Vg(t,\Vxi)$ admit time-dependent affine representations
$$
 \MA(t,\Vxi)=  \sum_{i=1}^{Q_\MA} \theta^i_\MA(t,\Vxi) \MA^i(t), \quad 
 \Vg(t,\Vxi) =  \sum_{i=1}^{Q_{\Vg}} \theta^i_{\Vg}(t,\Vxi)  \Vg^i(t).
%\Valp(0)  =\sum_{i=1}^{Q_{\Vu^0}} \theta^{\Vu^0}_i(t,\Vxi)   \boldsymbol{\alpha}_{i,r}(0,\Vxi) .
$$

\subsubsection{Solution of the reduced dynamical system}\label{subsub:reducflux}

The solution of the reduced dynamical system \eqref{eq:discretereduced}  requires an efficient evaluation of
$$
 \MA_r(t^k,\Vxi) =  \MV(t^k)^T \MA(t^k,\Vxi) \MV(t^k),\quad \Vg_r(t^k,\Vxi) =  \MV(t^k)^T\Vg(t^k,\Vxi), $$
 and 
 $$
  %\Vh_r(\Valp,t,\Vxi) = 
\Vh^{k+1}_r(\Valp(t^k,\Vxi),t^k,\Vxi)=  \MV(t^{k+1})^T \Vh(\Valp(t^k,\Vxi),t^k,\Vxi).
$$
%If $\MA$ and $\Vg$ admit time-dependent affine representations, then the reduced quantities $\MA_r$ and $\Vg_r$ admit time-dependent affine representations with the same time and parameter-dependent functions $\theta^{i}_{A}(t,\Vxi) $ and $\theta^i_g(t,\Vxi)$, and 
%$$
% \MA_r^i(t) =  \MV(t)^T\MA^i\MV(t), \quad
% \Vg_r^i(t) =  \MV(t)^T\Vg_r^i(t).
%% \boldsymbol{\alpha}_{r}^{0,i}(\Vxi) &=& \MV^T(0)\Vu^{0,i}.
%$$
In the offline phase, after computing the reduced space $X_r(t)$ and  associated basis $\MV(t)$, the reduced matrices $\MA_r^i(t)= \MV(t)^T\MA^i(t)\MV(t)$ and reduced vectors $\Vg_r^i(t)=  \MV(t)^T\Vg_r^i(t)$ can be precomputed. 
Then in the online phase, the reduced matrix $\MA_r(t,\Vxi)$ and reduced vector  $\Vg_r(t,\Vxi)$ can be evaluated for any parameters value and any time  
with a complexity independent of $d$ using
$$
 \MA_r(t,\Vxi)=  \sum_{i=1}^{Q_\MA} \theta^i_\MA(t,\Vxi) \MA_r^i(t), \quad 
 \Vg_r(t) =  \sum_{i=1}^{Q_{\Vg}} \theta^i_{\Vg}(t,\Vxi)  \Vg_r^i(t).
%\Valp(0)  =\sum_{i=1}^{Q_{\Vu^0}} \theta^{\Vu^0}_i(t,\Vxi)   \boldsymbol{\alpha}_{i,r}(0,\Vxi) .
$$
An efficient way of treating the nonlinear term is to approximate  
 $\Vh(\Vu_r(t,\Vxi),t,\Vxi)$ using the Empirical Interpolation Method (EIM).  Here, we briefly recall the  principle of this method (see \cite{Bebendorf2014}  for a detailed presentation). 
For a given basis $\{\Vh_1,\hdots, \Vh_m\}$ in $\RR^d$, with $m\le d$,
the EIM  approximation $\tilde \Vh(\Vu_r(t,\Vxi),t,\Vxi)$ of order $m$ of $\Vh(\Vu_r(t,\Vxi),t,\Vxi)$ is 
given by
 \begin{equation}
 \label{eq:interph}
\tilde \Vh(\Vu_r(t,\Vxi),t,\Vxi)
 = \MHm ( \MPm^T \MHm)^{-1} \MPm^T\Vh(\Vu_r(t,\Vxi),t,\Vxi),
% &=&\mathbf{U}_m \MPm^T  \Vh(\Vu_r(t,\Vxi),t,\Vxi)  , 
 % &=&  \MHm(t)( \MPm(t)^T \MHm(t))^{-1}  \Vh( \MPm(t)^T \MV(t)\Valp(t,\Vxi)) 
\end{equation}
where  $\MPm = [\Ve_{j_1},\dots, \Ve_{j_m}] \in \RR^{m\times d }$, with $\Ve_j$ the $j$-th unit vector in $\RR^d$, and  
$\MHm = [\Vh_1,\dots, \Vh_m] \in \RR^{d\times m }$.
Here, the vectors ${\Vh}_1,\hdots,{\Vh}_m$  are chosen as linear combinations of columns $(i_1,\hdots,i_m)$ of the matrix of snapshots $\Vh(\Vu(t^k,\Vxi^i),t^k,\Vxi^i)$ already computed during the greedy selection of  $\{\Vxi^1,\dots\Vxi^r\}$. The set of indices $\mathcal{I}_m = \{(i_l,j_l), 1\le l \le m\}$ is obtained with a greedy algorithm (see \cite[\S 4.4]{Bebendorf2014}) which ensures that the matrix $\MPm^T \MHm$  is non singular and well-conditioned. An approximation of the reduced flux $\MV(t^{k+1})^T\Vh(\Vu_r(t^k,\Vxi),t^k,\Vxi)$ is then given by
$$
\MV(t^{k+1})^T \mathbf{U}_m \MPm^T \Vh(\Vu_r(t^k,\Vxi),t^k,\Vxi),
$$ 
where $\mathbf{U}_m= \MHm ( \MPm^T \MHm)^{-1} \in  \RR^{m\times d } $. The matrices 
$\MV(t^{k+1})^T \mathbf{U}_m \in \RR^{r\times m}$ can be pre-computed during the offline phase and stored to be later multiplied by the vectors $\MPm^T \Vh(\Vu_r(t^k,\Vxi),t^k,\Vxi)$ which contains the components $(j_1,\hdots,j_m)$ of the flux $\Vh(\Vu_r(t^k,\Vxi),t^k,\Vxi).$ In practice, the EIM algorithm is stopped when the precision $\varepsilon$ is reached, i.e. when 
\begin{equation}
\label{eq:stoperrdeim}
\displaystyle \max_{k,i} \| \Vh(\Vu_r(t^k,\Vxi^i),t^k,\Vxi^i)- \tilde \Vh(\Vu_r(t^k,\Vxi^i),t^k,\Vxi^i) \|_X
< \varepsilon.
\end{equation}
In the following computations, the flux  approximation defined by \eqref{eq:interph} is  used for the solution of the discrete reduced dynamical system \eqref{eq:discretereduced}.  Here, we assume that $\varepsilon$ is chosen small enough for the interpolation error between $\tilde \Vh$ and $\Vh$ to be neglected in our error analysis. 

\begin{remark}
As in \cite{Wirtz2014}, this method can be reduced to the adaptive selection of the indices 
$(i_1,\hdots,i_m)$ only, where the set of vectors $\{\Vh_{1},\dots \Vh_{m}\}$  is obtained by a  POD  of the matrix containing the snapshots of the discrete flux \cite{Chaturantabut2010}. 
\end{remark}

\begin{remark}
In the present work, we have applied an EIM to construct  an approximation of
 the nonlinear flux $\Vh(\Vu_r(t,\Vxi),t,\Vxi)$ on a time-independent basis $\{\Vh_1,\hdots,\Vh_m\}$ in $\RR^d$. An alternative approach, not considered here, would consist in approximating the flux on a time-dependent basis $\{\Vh_1(t),\hdots,\Vh_m(t)\}$ constructed using an algorithm similar to the proposed T-greedy algorithm,  using a global in time error estimate providing indices $\mathcal{I}_m$.  
\end{remark}

\subsubsection{Evaluation of the error estimate} \label{sec:discreteerr}
Now we detail practical aspects for the evaluation of the a posteriori error estimate  $ \tilde \Delta_r$ using the integration scheme \eqref{eq:estimate2}. For an efficient online computation of the error estimate,   we need to 
evaluate with a complexity independent of $d$ the term 
$\|\tilde \Vr^{k}(\Vxi)\|_X$ and the 
estimations $\tilde L_X[\MA]$ and
 $\tilde L_X[\Vh] (\Vu_r(t,\Vxi))$ of the  logarithmic Lipschitz constants  $ L_X[\MA]$ and
 $ L_X[\Vh] (\Vu_r(t,\Vxi))$. \\

 We first address the estimation of logarithmic Lipschitz constants. Using the time-dependent affine  representation of $\MA$, we have%\footnote{In the particular case where $Q_\MA=1$ and $\theta_\MA^1\ge 0$, we let $\tilde L_X[\MA] = L_X[\MA]  = \theta_A^1(\Vxi)  L_X[\MA_1(t)]$.}
\begin{equation} \label{eq:approxlipA}
L_X[\MA] = \sup_{0\neq \Vx \in X}  \sum_{i=1}^{Q_\MA} \theta_\MA^i(t,\Vxi)
\dfrac{\langle\Vx,\MA^i(t)\Vx\rangle_X }{\|\Vx\|_X^2} \le \sum_{i=1}^{Q_\MA}  \vert \theta_\MA^i(t,\Vxi)\vert  \vert L_X[\MA^i(t)] \vert :=\tilde L_X[\MA].
\end{equation} 
For the local logarithmic constant of $\Vh$ at $\Vu_r$, we consider the first-order linearization  of the flux
 $$
 \Vh(\Vv,t,\Vxi) \approx \Vh(\Vu_r(t,\Vxi),t,\Vxi)+ \nabla\Vh(\Vu_r(t,\Vxi),t,\Vxi)(\Vv - \Vu_r(t,\Vxi)), %+ o(\|\Vx - \Vu_r(t,\Vxi)\|_X),
 $$
  where $\nabla\Vh(\Vu_r(t,\Vxi),t,\Vxi)\in \RR^{d\times d}$ denotes the gradient
 of $\Vh(\cdot,t,\Vxi)$ at $\Vu_r(t,\Vxi)$, 
and the corresponding approximation 
$$
L_X[\Vh](\Vu_r(t,\Vxi)) \approx L_X[\nabla\Vh(\Vu_r(t,\Vxi),t,\Vxi)].
$$ 
Then, a first order approximation of $L_X[\Vh](\Vu_r)$ is obtained by computing the largest eigenvalue of the symmetric part of  $\nabla\Vh(\Vu_r(t,\Vxi),t,\Vxi)$, which is a problem with complexity depending on $d$.  In  \cite{Wirtz2014}, in order to obtain a complexity independent on $d$, the authors use a partial similarity transformation of the matrix $\nabla\Vh$ based on POD (combined with 
a matrix-DEIM approximation) which preserves the largest eigenvalue of the symmetric part of $\nabla\Vh$. Here, we adopt a simpler strategy which consists in interpolating the Lipschitz constants of  the matrices $\nabla\Vh$ already computed during the solution of the full-order dynamical system in the offline step. We introduce a very simple  nearest neighbor interpolation 
\begin{equation} \label{eq:approxjac}
L_X[\nabla\Vh(\Vu_r(t,\Vxi),t,\Vxi)] \approx \sum_{i=1}^r \gamma^i(\Vxi)L_X[\nabla\Vh](\Vu(t,\Vxi^i),t,\Vxi^i) :=  \tilde L_X[\Vh(\Vu_r(t,\Vxi),t,\Vxi)],
 % := \tilde L_X[\Vh](\Vu_r),
\end{equation}
where the $\{\Vxi^i\}_{i=1}^r$ are the parameters values selected during the greedy procedure, and where 
$$
\gamma^i(\Vxi) = 
 \begin{cases}  1 & \displaystyle \text{ if } d(\Vxi,\Vxi^i) = \min_{1\le j \le r} d(\Vxi,\Vxi^j),  \\ 
 0 & \mbox{ otherwise},
 \end{cases} 
 $$
 with $d(\cdot,\cdot)$ a metric on $\VXi$ (typically the Euclidian metric).
%Thus, the upper bound of $\tilde L_X[\Vf](\Vu_r)$ in \eqref{eq:estimate2} is replaced by its approximation 
%$$
%\tilde L_X[\Vf](\Vu_r) = \tilde L_X[\MA(t,\Vxi)] +  \tilde L_X[\nabla\Vh(\Vu_r(t,\Vxi),t,\Vxi)].
%$$
Note that when using the interpolation method, $\tilde L_X[\Vh](\Vu_r) $ is not necessarily an upper bound of $L_X[\Vh](\Vu_r)$ and therefore,  the solution 
of \eqref{eq:errorEDO} with $L_X[\Vh](\Vu_r)$ replaced by $\tilde L_X[\Vh](\Vu_r)$ may not provide a certified error bound. However, as it will be seen in the numerical experiments, it provides a very good error indicator in practice.\\

Now we detail the computation of 
$\|\tilde \Vr^{k}(\Vxi)\|_X$ in an offline/online strategy, where 
$\tilde \Vr^{k}(\Vxi)$ is given by 
 \eqref{tilder_discrete} with $\tilde \Vh$ the approximation of $\Vh$ defined by 
 \eqref{eq:interph}.
 In what follows, the indices $i$ and $j$ take values in  $\{1,\hdots,Q_{\MA}\}$ or $\{1,\hdots,Q_{\Vg}\}$.
In the offline phase, we pre-compute the following matrices and vectors required for the online phase:
\begin{align*}
\mathbf{K}^{1,k}_{ij} &=   {\MV^k}^T {\MA^{i,k}}^T \MP_{{X_r^k}^\perp} \MA^{j,k}  \MV^k,
&\mathbf{K}^{2,k}_{i}  &=\frac{\delta \MV^{k}}{\delta t}^T\MP_{{X_r^{k+1}}\perp}\MA^{i,k+1} \MV^{k+1},\\
\mathbf{K}^{3,k} &=\frac{\delta \MV^k}{\delta t}^T\MP_{{X_r^{k+1}}^\perp}\frac{\delta \MV^k}{\delta t},
&\mathbf{K}^{4,k} &= \frac{\delta \MV^k}{\delta t}^T\MP_{{X_r^{k+1}}^\perp} \mathbf{U}_m,\\
\mathbf{K}^{5,k}_i &={\MV^k}^T {\MA^{i,k}}^T\MP_{{X_r^k}^\perp}\mathbf{U}_m,
\end{align*}
and
\begin{align*}
\boldsymbol{b}^{1,k}_{i}  &=   \frac{\delta \MV^k}{\delta t}^T \MP_{{X_r^{k+1}}^\perp}\Vg^{i,k},
&\boldsymbol{b}^{2,k}_{ij} & = {\MV^{k+1}}^T{\MA^{i,k+1}}^T\MP_{{X_r^{k+1}}^\perp}  \Vg^{j,k},\\
\boldsymbol{b}^{3,k}_{i} &=  \mathbf{U}_m^T \MP_{{X_r^{k+1}}^\perp}\Vg^{i,k}.
\end{align*}
Then, in the online phase we compute
\begin{align*}
 \|\tilde \Vr^k(\Vxi)\|^2_X
&=  \langle \Valp^{k+1}(\Vxi), \MM_1^{k+1}(\Vxi) \Valp^{k+1}(\Vxi) \rangle_X  
+  \langle \Valp^{k}(\Vxi), \MM_2^{k}(\Vxi) \Valp^{k+1}(\Vxi) \rangle_X  \\
& +  \langle \Valp^{k}(\Vxi), \MM_3^{k}(\Vxi) \Valp^{k}(\Vxi) \rangle_X   + \langle \Valp^{k}(\Vxi),\MM_4^k(\Vxi)\Vv_1^k(\Vxi)\rangle_X\\
 &  + \langle \Valp^{k+1}(\Vxi),\MM_5^{k+1}(\Vxi)\Vv_1^k(\Vxi)\rangle_X  + \langle \Vv_1^k(\Vxi),\MM_6^k\Vv_1^k(\Vxi)\rangle_X
 \\  & + \langle \Vv_1^k(\Vxi), \Vv_2^k(\Vxi)  \rangle_X+ \langle \Valp^k(\Vxi),  \Vv_3(\Vxi) \rangle_X 
+  \langle \Valp^{k+1}(\Vxi),  \Vv_4^k(\Vxi) \rangle_X  + b^k
\end{align*}
where the reduced quantities are defined by
$$
\MM_1^k(\Vxi)= 
\sum_{i,j=1}^{Q_{\MA}}  \theta_{\MA}^{i,k}(\Vxi)\theta_{\MA}^{j,k}(\Vxi) \mathbf{K}^{1,k}_{ij},\quad
\MM_2^k(\Vxi)  =  - 2  \sum_{i=1}^{Q_{\MA}}  \theta_{\MA}^{i,k+1}(\Vxi) \mathbf{K}^{2,k}_{i},  \quad 
\MM_3^k(\Vxi) =    \mathbf{K}^{3,k}, 
$$
$$
\MM_4^k(\Vxi)  =   -2 \mathbf{K}^{4,k},  \quad 
\MM_5^k(\Vxi) =  2  \sum_{i=1}^{Q_{\MA}}  \theta_{\MA}^{i,k}(\Vxi) \mathbf{K}^{5,k}_{i} , \quad
\quad \MM_6^k    =    \mathbf{U}_m^T \MP_{{X_r^k}^\perp} \mathbf{U}_m,
$$
\begin{align*}
\Vv_1^k(\Vxi) & =   \mathbf{P}^T_m \Vh^k(\Vu_r^k(\Vxi),\Vxi),
& \Vv_2^k(\Vxi) &= 2 \sum_{i=1}^{Q_{\Vg}} \theta_{\Vg}^{i,k}(\Vxi) \boldsymbol{b}^{3,k}_i,\\
\Vv_3^k(\Vxi) &= -2 \sum_{i=1}^{Q_{\Vg}}  \theta_{\Vg}^{i,k}(\Vxi) \boldsymbol{b}_{i}^{1,k}, \quad 
&\Vv_4^k(\Vxi) &= 2  \sum_{i=1}^{Q_{\MA}} \sum_{j=1}^{Q_{\Vg}}  \theta_{\MA}^{i,k+1}(\Vxi)  \theta_{\Vg}^{j,k}(\Vxi) \boldsymbol{b}_{ij}^{2,k},
\end{align*}
and 
$$
 b ^k = \sum_{i,j=1}^{Q_{\Vg}}  \theta_{\Vg}^{i,k}(\Vxi) \theta_{\Vg}^{i,k}(\Vxi) \langle \Vg^{i,k},\Vg^{j,k}\rangle_X. 
$$

  %%%%%%%%%%%%%%%%%%%%%%%%%%%
 %%% NUMERICAL EXPERIMENT 
 %%%%%%%%%%%%%%%%%%%%%%%%%%%
 
 \section{Numerical experiments}\label{sec:numres}
 
 In this section, we present numerical results where 
we compare model order reduction methods using either time-independent or time-dependent reduced spaces. Also, we evaluate the effectiveness of the a posteriori  error estimate derived in Subsection \ref{subsec:error}. 
We denote by MTI (resp. MTD) the model order reduction method using time-independent (resp. time-dependent) reduced spaces. If not mentioned, reduced spaces for MTI (resp. MTD) are constructed with a POD-greedy algorithm (resp. T-greedy algorithm) using a training set $\VXi_{train}$  whose size will be specified in each case. 
 \\
 
 {
 We consider three test cases which are particular cases of the following nonlinear partial differential equation defined on a domain $\Omega$ which is a domain in $\RR^s$ (with $s=1,2$) and a time interval $I=(0,T)$:
 \begin{equation} \label{eq:qsl}
\frac{\partial}{\partial t} u + \nabla \cdot (c(u,\Vxi ) u) + \Va(\Vx,\Vxi) \cdot \nabla u
- \mu(\Vxi) \Delta u  =   g(\Vx,t,\Vxi), \quad \text{on} \quad \Omega \times I,
\end{equation}
with appropriate boundary conditions and a parameter-independent initial condition $u^0(\Vx)$. We denote $\Vx = (x_1,\dots,x_s)^T$ the spatial variable and the corresponding spatial differential operators $\nabla = (\frac{\partial}{\partial x_1},\dots, \frac{\partial}{\partial x_s})^T$ and $\Delta = \sum_{i=1}^s \frac{\partial^2}{\partial x_i^2}$. Here $\Vxi = (\xi^1,\dots, \xi^p)$ denotes a random vector with values in $\VXi$ and with independent components. The functions $c: X \times \VXi \to \RR$ and $\Va : \Omega \times \VXi \to \RR^s$ will be specified in each test case. Finally,  
$\mu : \VXi \to \RR$ is a parameter-dependent coefficient and $ g : \Omega \times I \times \VXi \to \RR$ is a given source term.
We consider an approximation  of $u$ obtained with an appropriate scheme (e.g. finite differences, finite element) depending on the test case. This yields a system of $d$ ordinary differential equations of the form \eqref{eq:dynsystgen}. Here, $d=d(n)$ corresponds to the dimension of the discrete problem taking into account the boundary conditions (e.g. when using finite differences for one-dimensional problem, $d=n-2$ for Dirichlet conditions, or $d=n-1$ for periodic conditions). 
Finally, we denote the spatial approximation of the solution as follows $\Vu(t) = (u_i(t))_{i=1}^{d} \in \RR^d$, with approximations of both initial condition $\Vu^0 \in \RR^d$, and source term $\Vg(t,\Vxi) \in \RR^d$ that will be specified later. We assume that the time integration scheme is accurate enough for the error due to this approximation to be neglected. \\

Numerical experiments were conducted with an in-house code written in Matlab\textsuperscript{\textregistered}. 
}

\subsection{Test case 1}\label{subsec:t1}
{
Let $\Omega= (0,1)$ and $I=(0,0.2)$. We consider an advection equation with 
$\mu(\Vxi)=0$,   $g(x,t,\Vxi)=0$, $c(\Vu,\Vxi)=0$ and $a(x,\Vxi ) =a(\Vxi)= a_0+a_1\xi^1$, with 
$\xi^1 \sim U(-1,1)$,
 $a_0 =1$ and $a_1=0.5$. The initial condition is a smooth function given by $u_{cont}^0(x) = \frac{1}{\sqrt{2\pi}}\exp(-\left( \frac{x-0.6}{0.05}\right)^2)$. 
We consider an approximation  of $u$ obtained with an appropriate finite difference scheme over a uniform discretization $\{x_i\}_{i=1}^{n}$ of $\Omega$,  where $x_i=a + i \delta x$, $\delta x = \frac{b-a}{n-1}$. 
This yields a system of $d$ ordinary differential equations of the form \eqref{eq:dynsystgen}. 
We impose periodic boundary conditions and consider a finite difference
upwind scheme  on a uniform discretization with {$n = 2001$ points}. We use an explicit Euler time integration scheme 
\begin{equation}
\label{eq:pbadv}
\Vu^{k+1}(\Vxi)=(\MI_d  + \delta t \MC(\Vxi))  \Vu^k(\Vxi),
\end{equation}
with initial condition $\Vu^0= (u^0(x_i))_{i=1}^{d}$. Here $\delta t = 0.5 \frac{\delta x}{{c_0+c_1}}$,  where 
$\mathbf{C}(\Vxi) \in \RR^{d \times d} $ corresponds to the discrete 
advection operator with periodic boundary conditions.
}
\paragraph{Deterministic case}
We consider a deterministic problem with a fixed value 
$\Vxi=\Vxi_0=0.65$.
Here, we compare the approximations obtained by projections on reduced spaces constructed in two different ways.  In the first method (MTI), the reduced space $X_r$ is  time-independent and generated by the first $r$ modes of the POD of the trajectory $t \mapsto \Vu(t,\Vxi_0)$, with $r \in \{1,2,5,10,20,50,100,200 \}$. In the second method (MTD), we consider the one-dimensional time-dependent space 
$X_r(t)=\spann\{\Vu(t,\Vxi_0)\}$, $r=1$.
 
 \begin{figure}[H]
\centering
 \includegraphics[height=4.2cm]{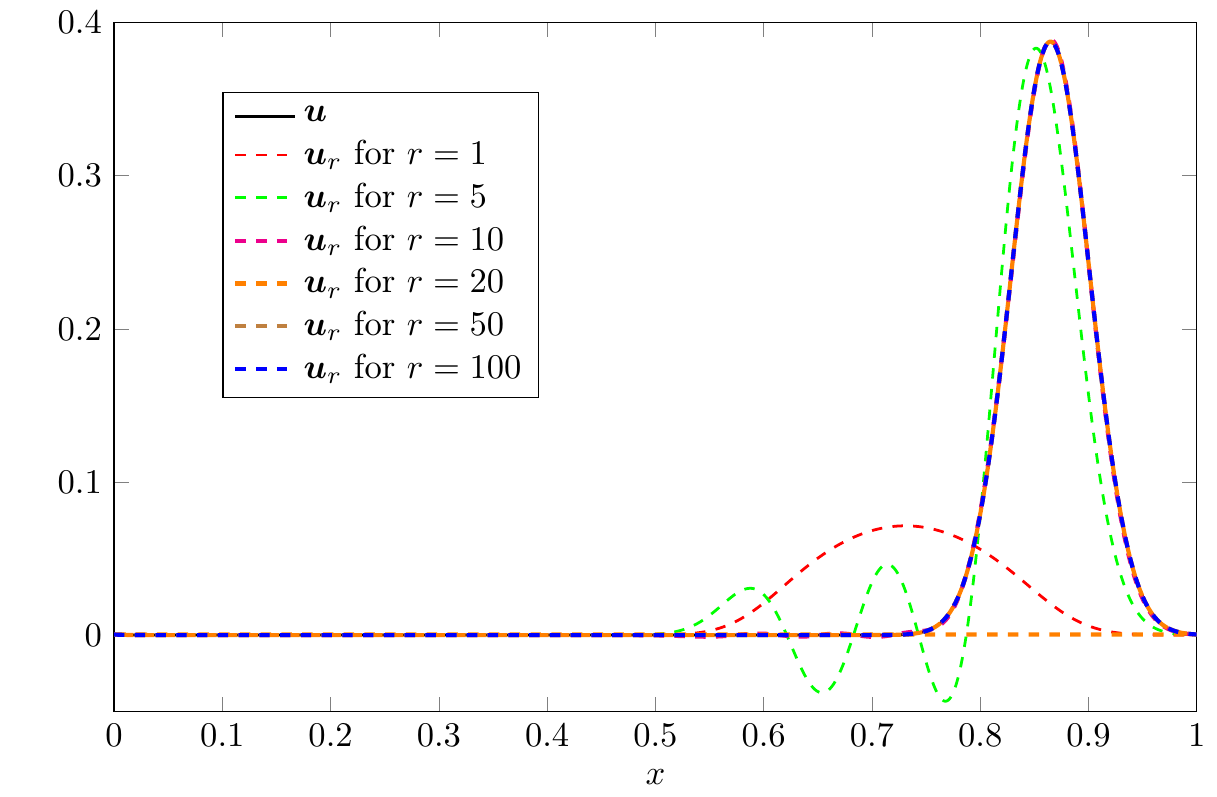}
  \includegraphics[height=4.2cm]{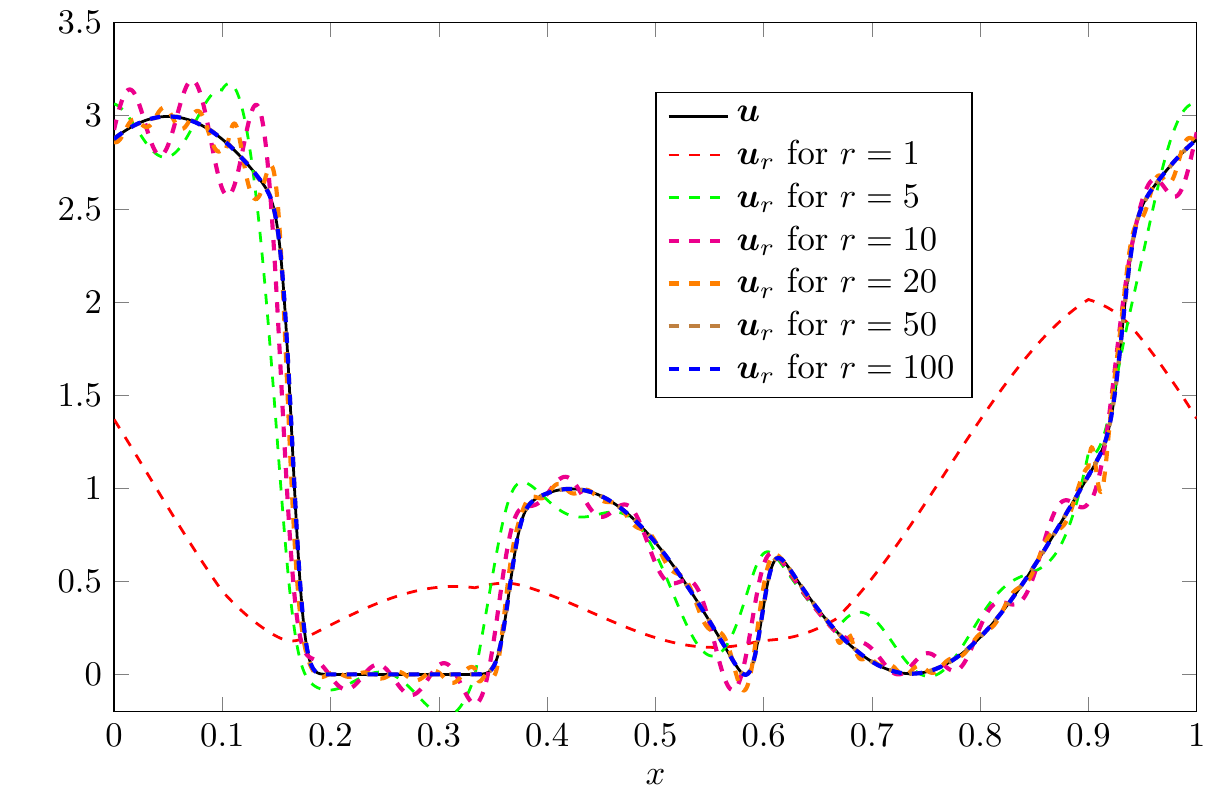}
   \caption{Test case 1: comparison for discontinuous (left) and continuous (right) initial conditions at final times of the exact solution $\Vu$ and the approximation $\Vu_r$
  computed with MTI for $r \in  \{1,5,10,20,50,100 \}$.
 \label{fig:ADV1}}
 \end{figure}
 
 \begin{table}[H]
\centering
{
 \begin{tabular}{|c|c|c|}
  \multicolumn{3}{c}{Relative errors for $u_{cont}^0$}\\
\hline
$\dim(X_r) $ & {$E_2$}  & $E_\infty$\\
\hline   \hline
$1$   & $0.027831$ 			& $0.76058$\\
$2$   & $ 0.024075$		        &  $  0.95713$\\
$5$   & $0.0078461$		        &  $ 0.26984$\\
$10$ &   $ 0.00021853$		&  $ 0.0075156$ \\
$20$ & $ 4.8616e-09$			&  $1.6719e-07$ \\
$50$ & $4.0924e-17$		& $1.4074e-15$\\
$100$ & $4.4353e-17$		& $1.5253e-15$ \\
$200$ & $ 4.5623e-17$		& $1.569e-15$ \\
   \hline
\end{tabular}
 \begin{tabular}{|c|c|c|}
 \multicolumn{3}{c}{Relative errors for $u_{disc}^0$}\\
\hline
$\dim(X_r) $ & {$E_2$}  & $E_\infty$\\
\hline   \hline
$1$   & $0.022116$ 			& $0.76058$\\
$2$   & $0.0135$		        &  $0.46427$\\
$5$   & $0.0060424$		        &  $0.2078$\\
$10$ &   $0.0040044$		&  $0.13771$ \\
$20$ & $ 0.002043$			&  $0.07026$ \\
$50$ & $4.7139e-05$		& $0.0016211$\\
$100$ & $ 1.3891e-12$		& $4.7771e-11$ \\
$200$ & $ 6.3966e-17$		& $2.1998e-15$ \\
   \hline
\end{tabular}
}
  \caption{Test case 1: relative errors $E_2$ and $E_\infty$ obtained with the MTI with respect to $\dim(X_r)$ for both continuous (left) and discontinuous (right) initial conditions conditions.  \label{tab:ADV1}}
\end{table}

{As shown onFigure \ref{fig:ADV1}, the approximations obtained with MTI are satisfactory when the dimension of the  reduced 
space $X_r$ is greater than $20$. 
This is confirmed by the values of the relative  errors $E_q = \|\Vu_r-\Vu\|_{I,q}/ \|\Vu\|_{I,q}$ given in Table \ref{tab:ADV1} for $q=2$ and $q=\infty$, where $\| \Vu\|_{I,q}$ is the natural norm on $L^q(I;X)$.
In particular, the relative error in $L^2$-norm (resp $L^\infty$-norm) is of order $10^{-17}$ (respectively  $10^{-15}$) for a space $X_r$ with a reduced dimension of $50$. As expected, MTD, with a reduced space $X_1(t)$ of dimension 1 containing the exact solution, gives relative errors at the machine precision ($E_2 =5.7\, 10^{-17}$ and $E_\infty=1.3\, 10^{-15}$)
For obtaining a very accurate precision, the method MTI requires a reduced space with rather high dimension, thus highlighting the limits of the POD method for  transport equations. Now, we illustrate the impact of the smoothness of the solution. For this purpose, we consider the same test case with a discontinuous initial condition given by  $u_{disc}^0(x)=\Ind_{[0.1,0.9]}(x)\cdot( \lfloor 3 x \rfloor + \sin(10x))$ 
\footnote{ 
Here, $\Ind_A$ is the characteristic function of a subset $A\subset \RR^d$ and $\lfloor x \rfloor$ denotes the integer part of $x$.}. 
The approximate solution is plotted onFigure \ref{fig:ADV1}  and the relative errors between the reduced approximation and the exact solution are summarized in the right table of Table \ref{tab:ADV1}. Here, we clearly observe that for non smooth initial condition $u_{disc}^0$ a  reduced space with higher dimension is needed for well approximating the solution (e.g. dimension 200 against 50 for reaching the machine precision). Concerning the MTD, we still obtain relative errors up to the machine precision with a one-dimensional time dependent reduced space ($E_2 =3.9\,10^{-17}$ and $E_\infty=1.3\,10^{-15}$).% . In what follows, we pursue our numerical experiments with $u_{disc}^0$ as initial condition.
}\\

\paragraph{General case} 
We now consider the parameter-dependent problem and compare the approximations obtained with MTI and MTD for a subspace of dimension $r=20$ generated from a training set of size $30$. 
Figure \ref{fig:ADV0} plots the evolution with time of the approximation obtained with MTD evaluated at $\Vxi=0.65$. In Figure \ref{fig:ADV2}, the exact solution is compared to the approximations obtained by MTI and MTD, at final time and for $\Vxi=0.65$. The  solutions of the reduced order models are very close to the exact solution. Nevertheless, 
we notice that the approximation computed with MTI present  small oscillations whereas the one obtained with MTD  matches very well the exact solution (see right plot ofFigure \ref{fig:ADV2}).\\ 
 
    \begin{figure}[H]
 \centering
 \includegraphics[height=4.75cm]{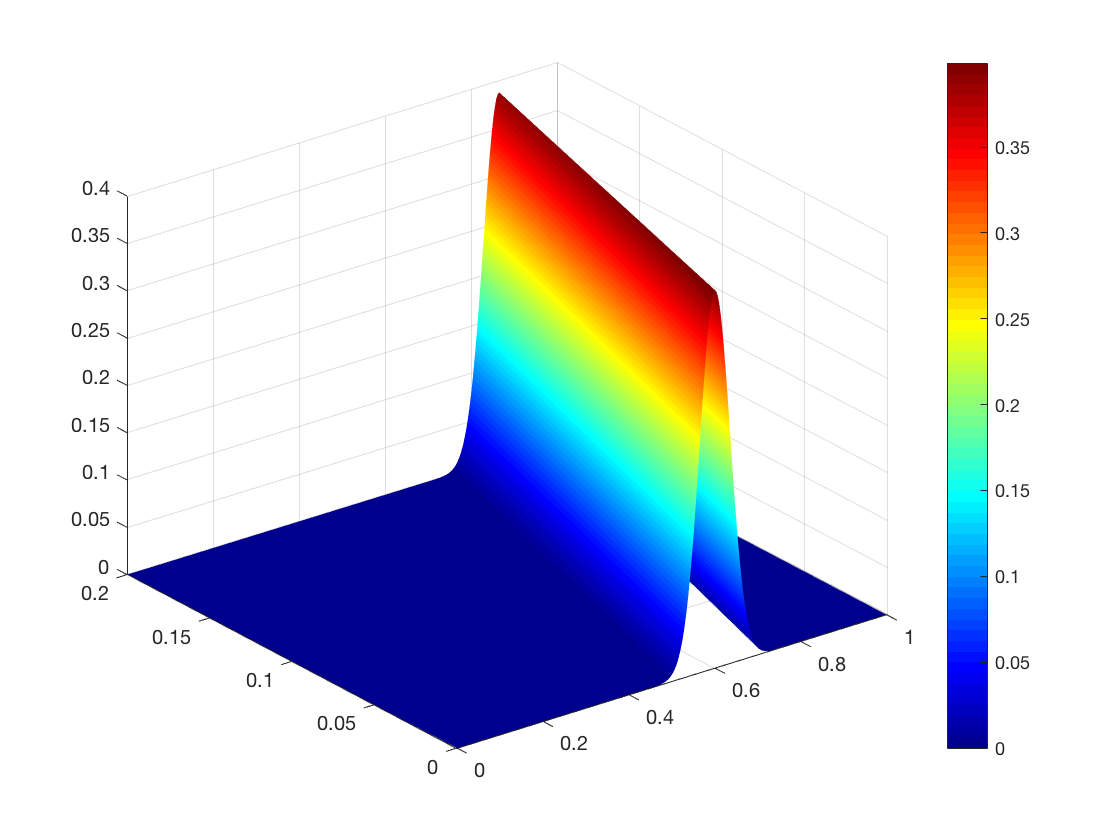}
 \includegraphics[height=4.75cm]{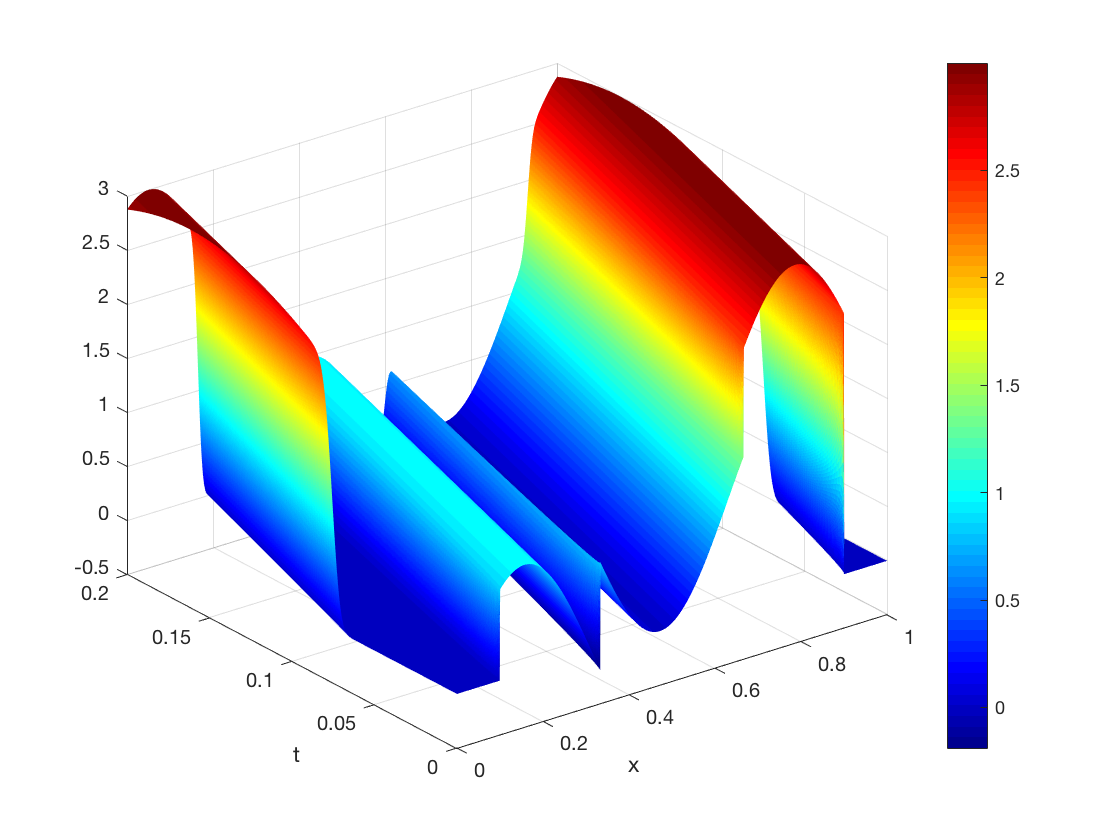}
 \caption{Test case 1: Time evolution of $\Vu_r$ computed with MTD for $r=20$ and $\Vxi=0.65$.}
 \label{fig:ADV0} 
 \end{figure}
  
      \begin{figure}[H]
 \centering
 \includegraphics[height=4.1cm]{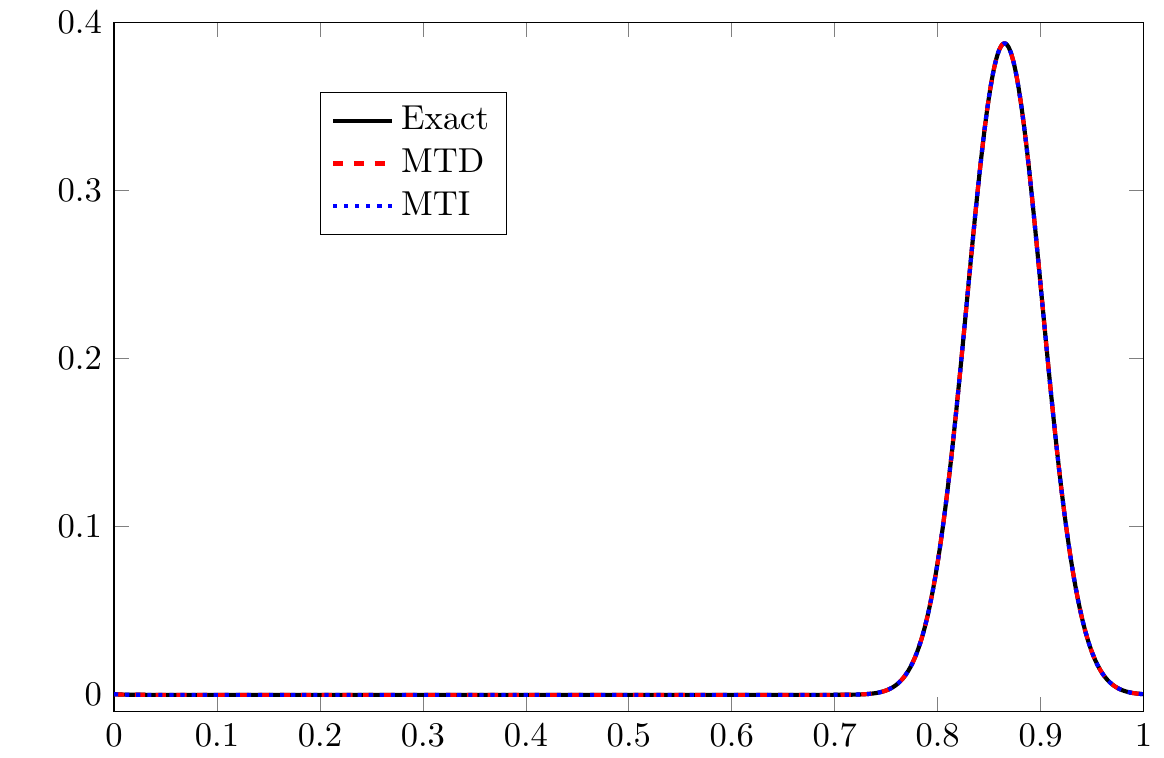} 
 \includegraphics[height=4.1cm]{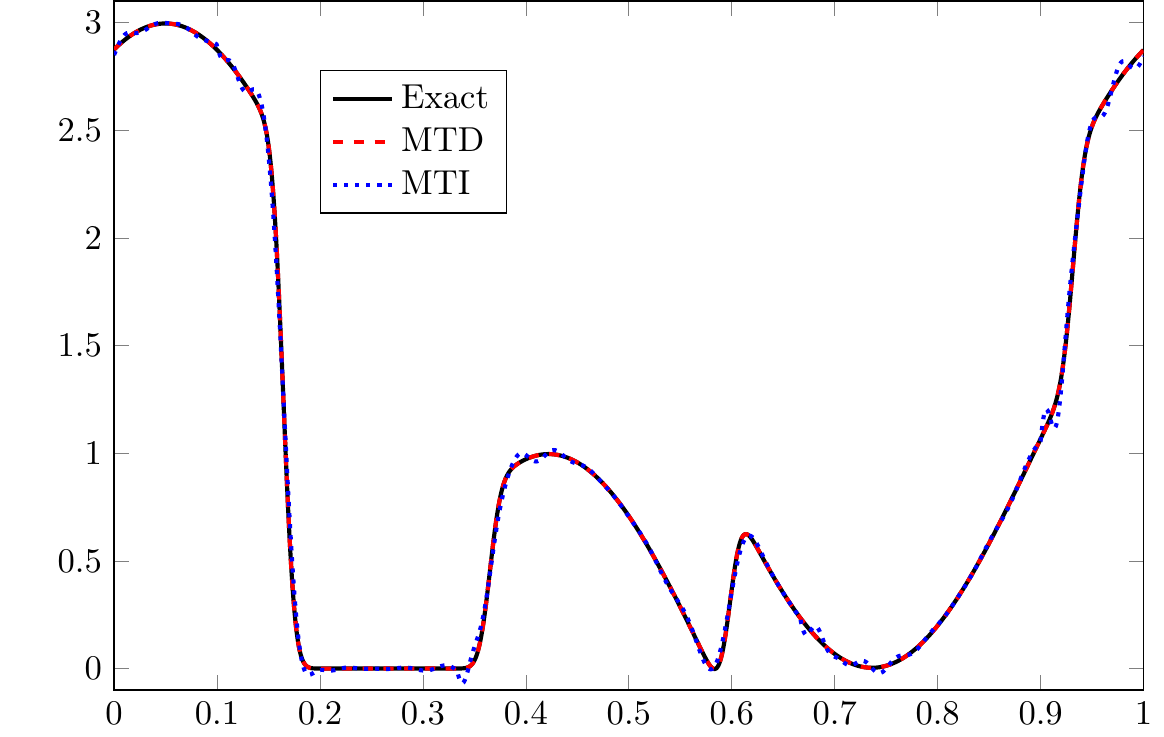} 
\caption{Test case 1: Reduced approximations $\Vu_r$
provided by MTD  and MTI compared to the exact solution $\Vu$ for both continuous (left) and discontinuous (right) initial conditions, at final time $T=1$ for $r=20$ and $\Vxi=0.65$.}
 \label{fig:ADV2} 
 \end{figure}
 
Then, 
we estimate the expectation $\Exp(E_q(\Vxi))$ and maximum $\max_{\Vxi \in \VXi}(E_q(\Vxi))$ of the relative errors $E_q$ respectively by the empirical mean and by the maximum of the values of $E_q$ taken at $50$ randomly chosen values of the parameters. 
These quantities are depicted onFigure \ref{fig:ADV5} for different values of $r$ and for both continuous and discontinuous initial conditions. { For the same dimension of the reduced spaces, MTD  clearly provides a more accurate approximation than 
 MTI in particular when considering discontinuous initial condition. %For $r=20$, we reach a precision of order $10^{-5}$ with MTD and of order $10^{-2}$ with MTI.
}
\begin{figure}[H]
 \centering
 \includegraphics[height=4cm]{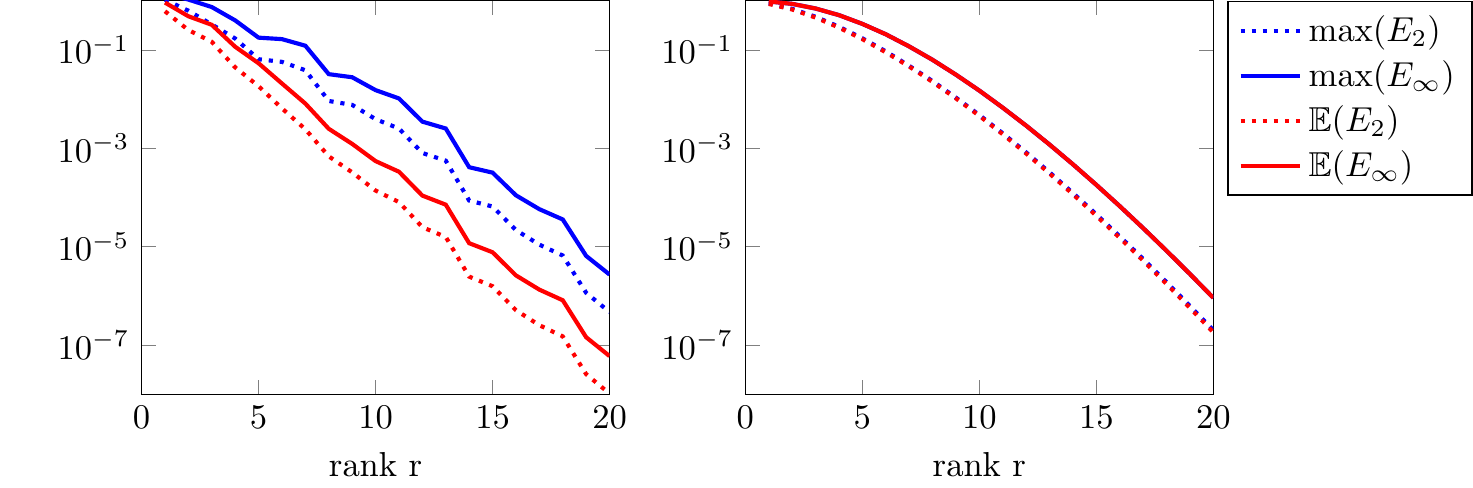} 
 \includegraphics[height=4cm]{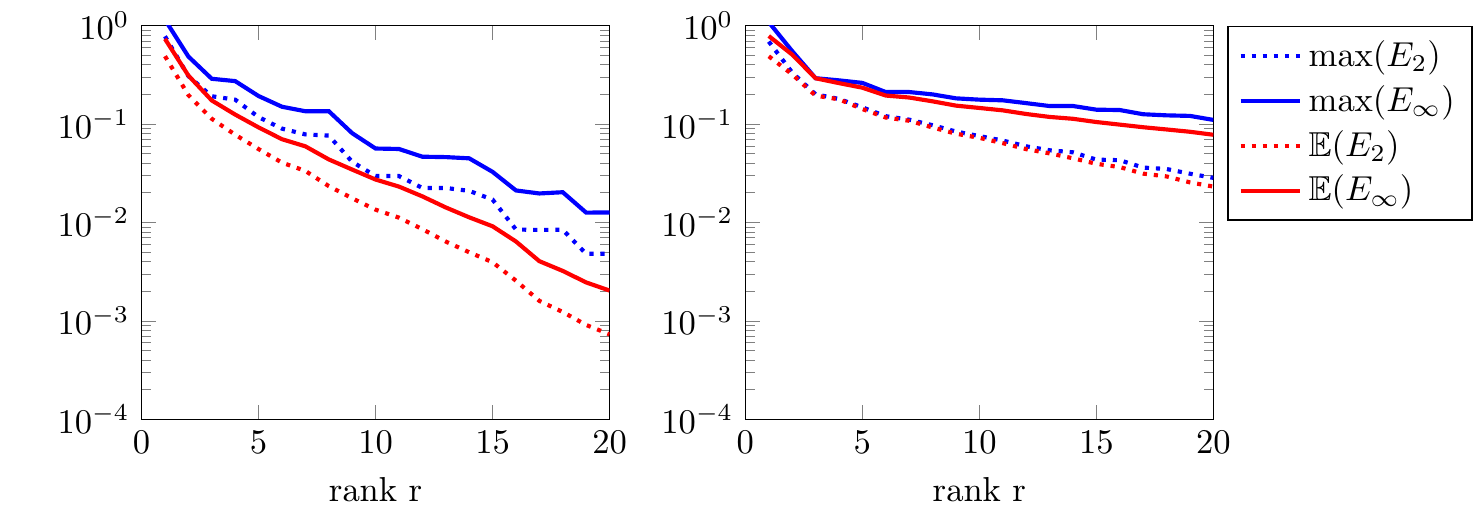} 
\caption{Test case 1:  Statistical estimations of the expectation and maximum of the relative errors $E_2$ and $E_\infty$ with respect to the reduced dimension $r$, for  MTD (left) and MTI (right)  for both continuous (top) and discontinuous (bottom) initial conditions. }
 \label{fig:ADV5} 
 \end{figure}
 
 \paragraph{Validation of the error estimate}
  We now investigate the efficiency of the proposed error estimate $\tilde \Delta_r$.Figure \ref{fig:ADV4} shows the evolution with time of 
$\tilde \Delta_r(t,\Vxi)$ and the exact error $\|\Ve_r\|_X$ for MTD and MTI, evaluated  at $\Vxi=0.65$. As expected, we observe that the error estimates have an exponential behavior in time. The error estimate $\tilde \Delta_r$ is much sharper for MTD than for MTI. For MTI, $\tilde \Delta_r$ is a very pessimistic upper bound of the error. 
 {This observation is confirmed by computing the effectivity index $\kappa(t,\Vxi) =\tilde \Delta_r(t,\Vxi)/\|\Ve_r(t,\Vxi)\|_X$. Figure \ref{fig:ADVDIF5} plots the evolution with time of a statistical estimation of  the mean of the effectivity index $\Exp(\kappa(t,\Vxi))$, which  remains close to $1$ with MTD and takes high values with MTI.}
 
   \begin{figure}[H]
 \centering
  \includegraphics[height=3.5cm]{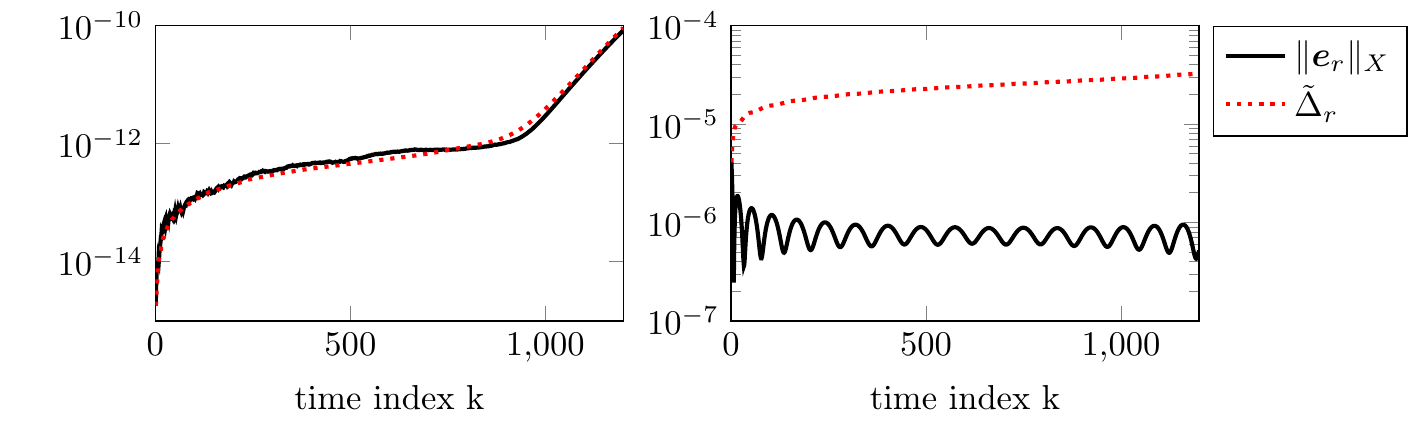}
    \includegraphics[height=3.5cm]{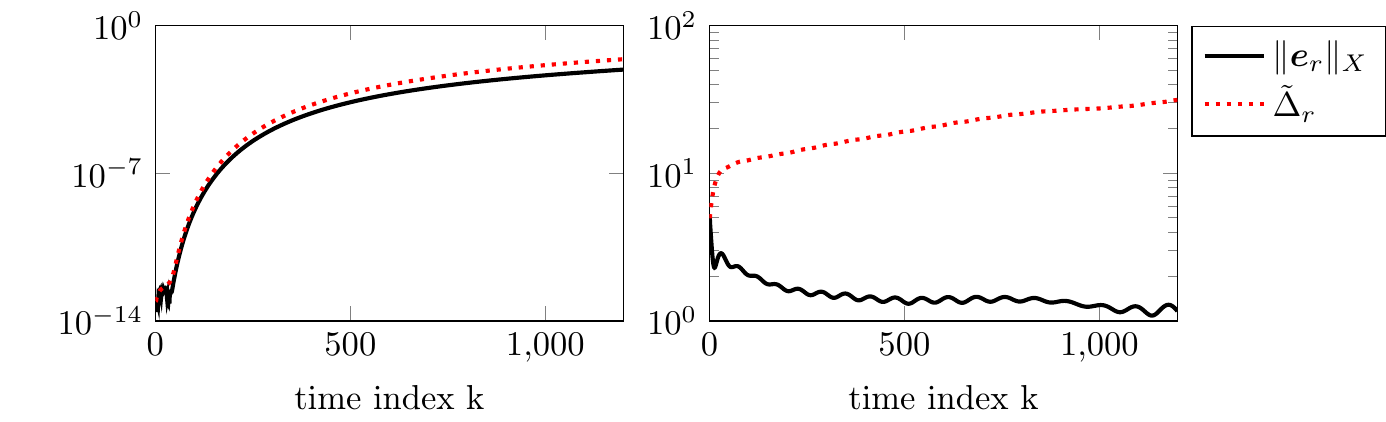}
  \caption{Test case 1: Evolution with time of $\tilde \Delta_r(t,\Vxi)$ and $\|\Ve_r(t,\Vxi)\|_X$ for  MTD (left) and MTI (right) for both continuous (top) and discontinuous (bottom) initial conditions, for $r = 20$ and $\Vxi=0.65$.} \label{fig:ADV4}
 \end{figure}

\begin{figure}[H]
 \centering
  \includegraphics[height=4cm]{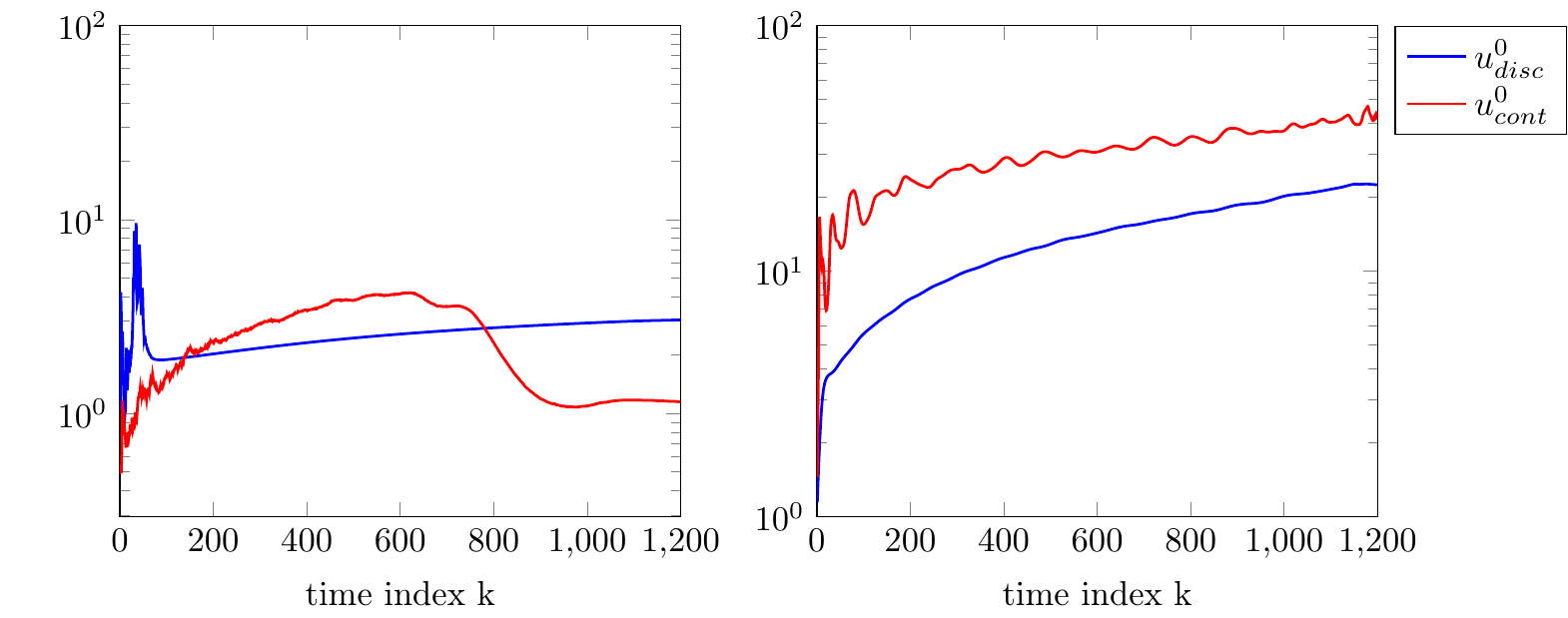}
  \caption{Test case 1: Evolution with time of a statistical estimation of $\Exp(\kappa(t,\Vxi))$ for MTD (left) and MTI (right) for both continuous and discontinuous initial conditions. }\label{fig:ADV3}
 \end{figure}

 {
 \paragraph{CPU times} 
Now, we briefly discuss the CPU computational costs of both MTI and MTD methods with POD-greedy and T-greedy algorithms respectively. We restrict the presentation to the case of the initial condition $u^0_{disc}$ (similar results are observed for $u^0_{cont}$). The dimension of reduced spaces is fixed to $r=20$ for both methods. As summarized in Table \ref{tab:ADVCPU}, both offline and online costs are roughly of the same order but the MTD provides a reduced order model which is by $2$ orders of magnitude (see Figure  \ref{fig:ADV5}) better than the reduced order model provided by MTI. Yet, we notice that for the MTI, CPU-times are slightly higher. This is probably due to additional computations involved by time-dependent reduced quantities in MTD. We may guess that the CPU times for both methods would be similar for non autonomous problems with time affine dependent coefficients.
}

\begin{table}[H]
 \centering
% discontinu
  \begin{tabular}{|c|c|c|}
%   \multicolumn{3}{c}{With $u_{disc}^0$} &  \multicolumn{3}{c}{With $u_{cont}^0$}\\
   \hline
& Offline & Online \\ %& Offline & Online\\
\hline   \hline
MTD & $1955.1$ & $11.24$ \\% & POD-Greedy & $1756.8$ & $9.05$\\
MTI & $1517.5$  &  $8.94$ \\%& T-Greedy & $ 2053.4$  &  $13.51 $\\
   \hline
\end{tabular}
  \caption{Test case 1: Offline and online CPU-times for both MTD and MTI, with $u_{disc}^0$.} \label{tab:ADVCPU}
\end{table}
 
\subsection{Test case 2}

{We now consider a two-dimensional advection-diffusion equation with  source term $g = 0 $,  $c(u,\Vxi) = 0$, $\mu(\Vxi) = \mu^0(2+\cos( \pi\xi^1)^2)$, with $\mu^0 = 0.5$,  and $a(\Vx,\Vxi )=a(\Vxi) (x_2-0.5, 0.5-x_1)^T$, with $a(\Vxi) = a_0\sin(\pi\xi^2)$ and $a_0 =0.1$. Here, we choose 
 $\xi^1,\xi^2 \sim U(-1,1)$ as independent uniform random variables. The spatial domain is $\Omega = (0,1)^2$ and the time interval $I = (0,0.2)$. The initial condition is given by $u^0(x) = e^{-(x_1-\frac{2}{3})^2-(x_2-\frac{2}{3})^2}\sin(2\pi x_1)\sin(2\pi x_2)$ and we impose homogeneous Dirichlet boundary conditions. We consider a $\mathbb{P}_1$-Lagrange finite element discretization with $n=1681$ nodes \footnote{The mesh is chosen fine enough to ensure a P\'eclet number smaller than 1.} together with an implicit Euler scheme with $K=400$ time steps, yielding the following scheme
\begin{equation}
(\MI_X - \delta t \MA(\Vxi)) \Vu^{k+1}(\Vxi)= \Vu^k(\Vxi),
 \label{eq:implicit}
\end{equation}
with $\MA(\Vxi) =  \mu(\Vxi) \MA_D + a(\Vxi)\MA_C  \in \RR^{ d \times d }$  ($d=1521$), where  $\MA_C$ and $\MA_D$  are the product of the mass matrix inverse with the discrete two dimensional diffusion and convection operators, obtained by finite element approximation,  respectively. \\

 We first present the reduced approximations computed by MTI and MTD with reduced spaces of dimension $r=30$ selected with greedy algorithms using a training set of size $60$. 
Figure \ref{fig:ADVDIF1} represents the approximation obtained with MTD evaluated for different values of the parameter $\Vxi \in \{ (0.2,-1),(0.2,0),(0.2,0.5)\}$. 
 
 \begin{figure}[H]
 \centering
 \includegraphics[height=3cm]{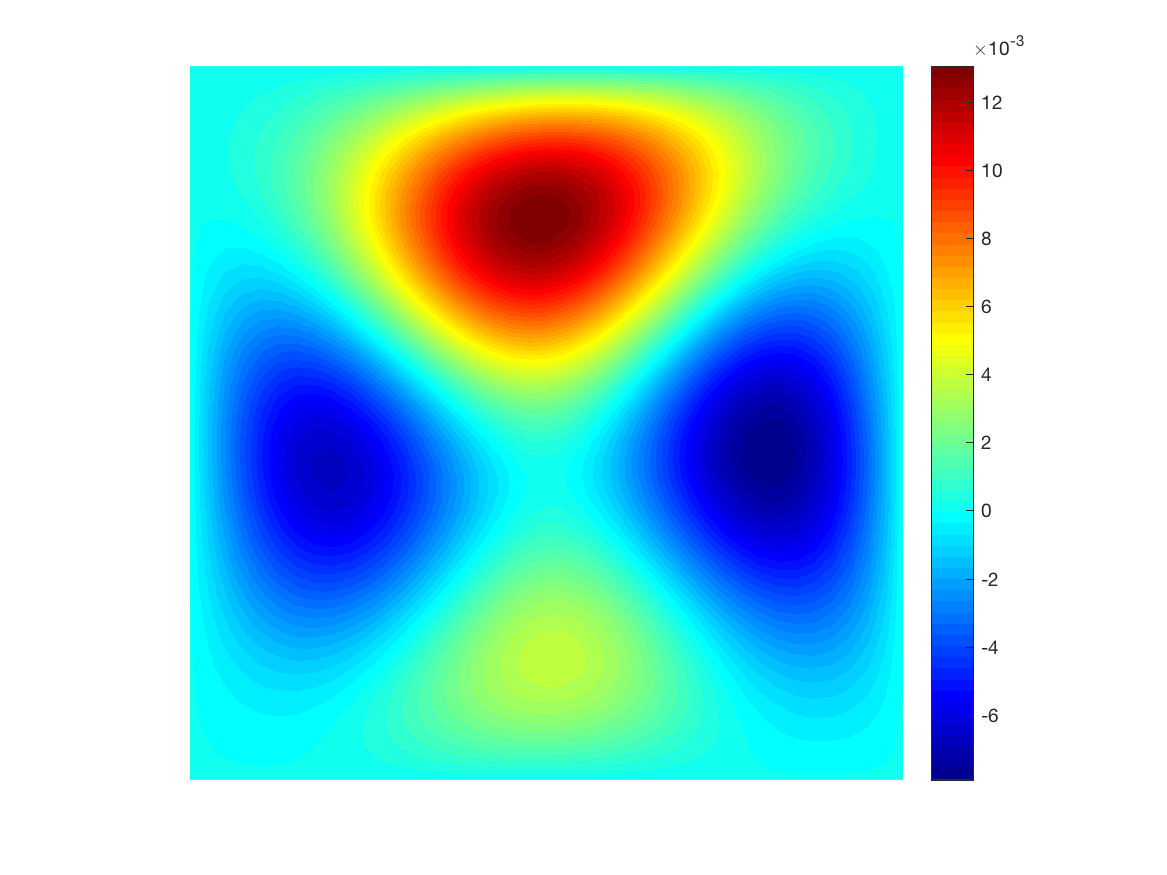}
\includegraphics[height=3cm]{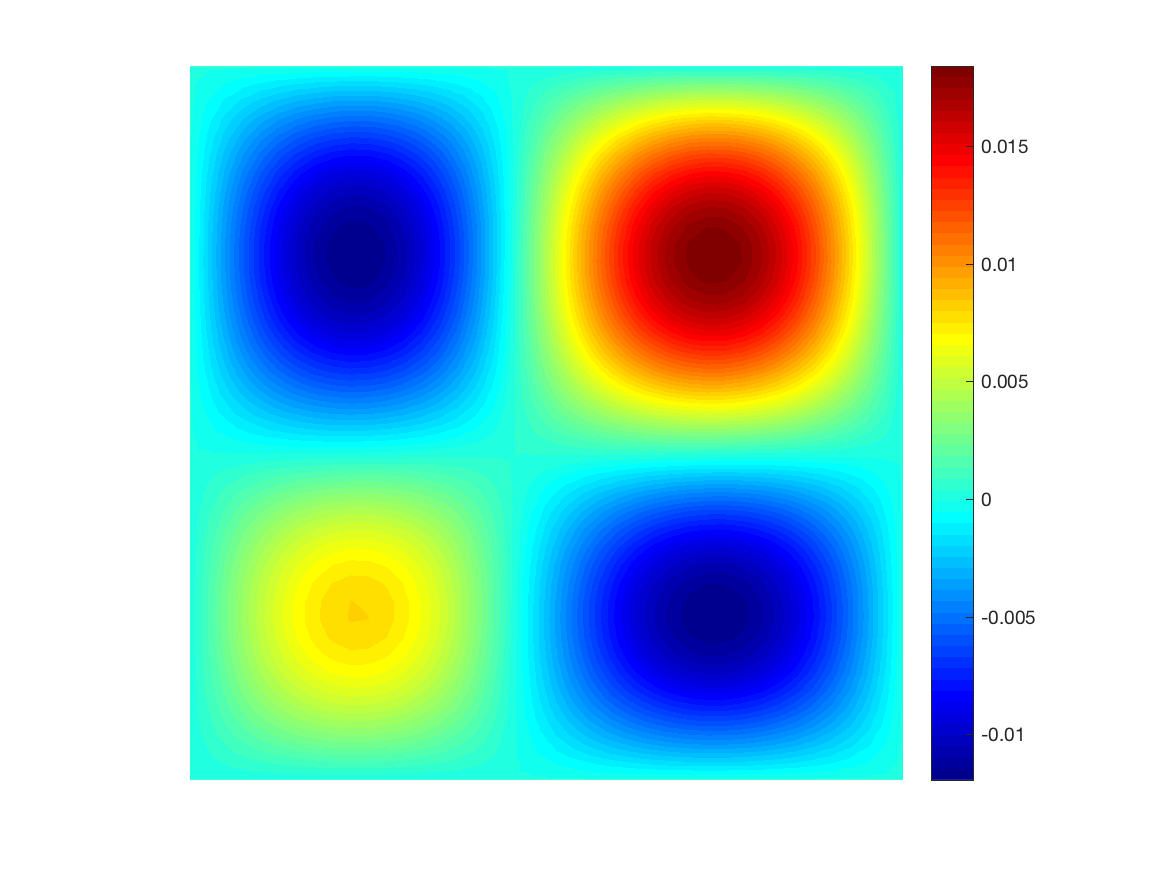}
\includegraphics[height=3cm]{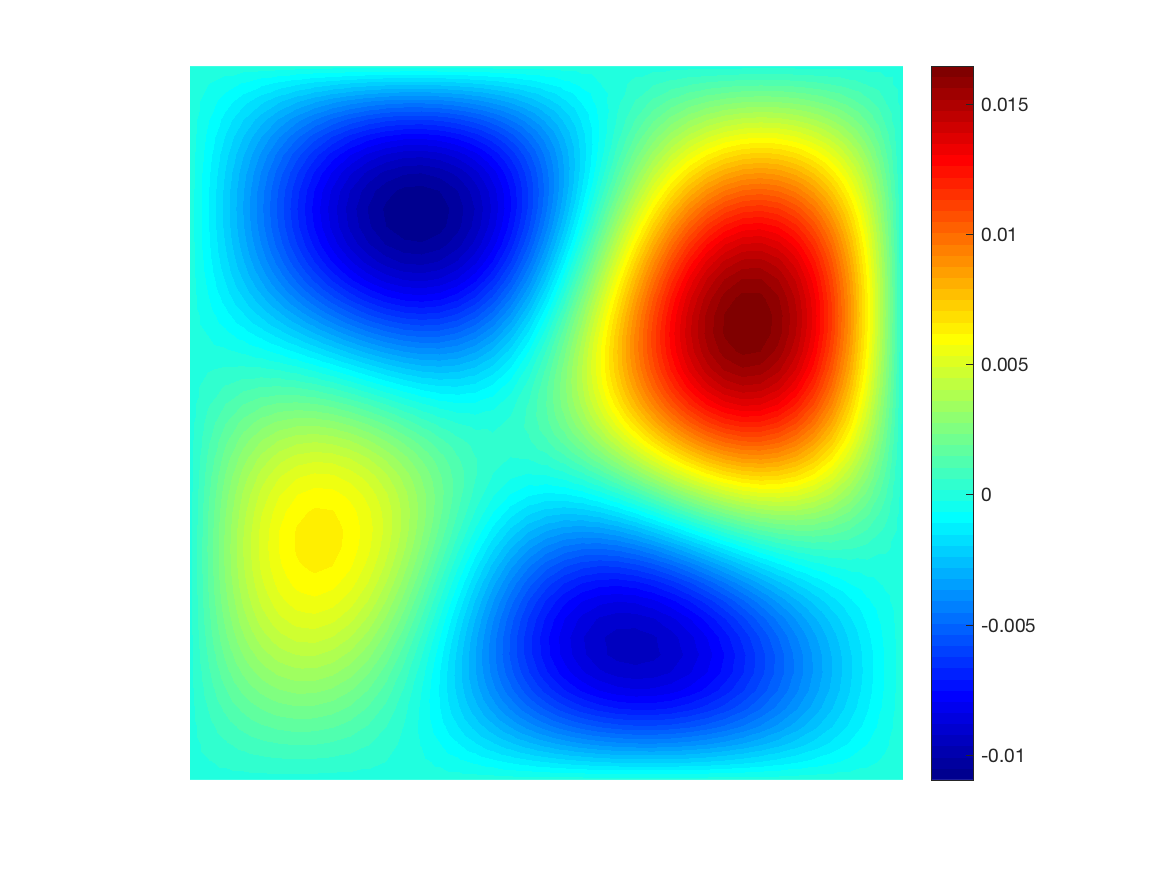}
 \caption{Test case 2: reduced approximation $\Vu_r$ computed with MTD for $r=30$ at final time and $\Vxi \in\{ (0.2,-1),(0.2,0),(0.2,0.5)\}$.}
 \label{fig:ADVDIF1}
 \end{figure}
   For the same parameter values, the absolute distance to the exact solution at final time  is given  onFigure \ref{fig:ADVDIF1} for MTI and MTD. It shows that the approximations obtained by both order reduction methods are in good agreement with the exact solution. Once again, we observe that MTD provides a better approximation with an error of order $10^{-9}$ against $10^{-5}$ for MTI.
 
\begin{figure}[H]
 \centering
 \includegraphics[height=3cm]{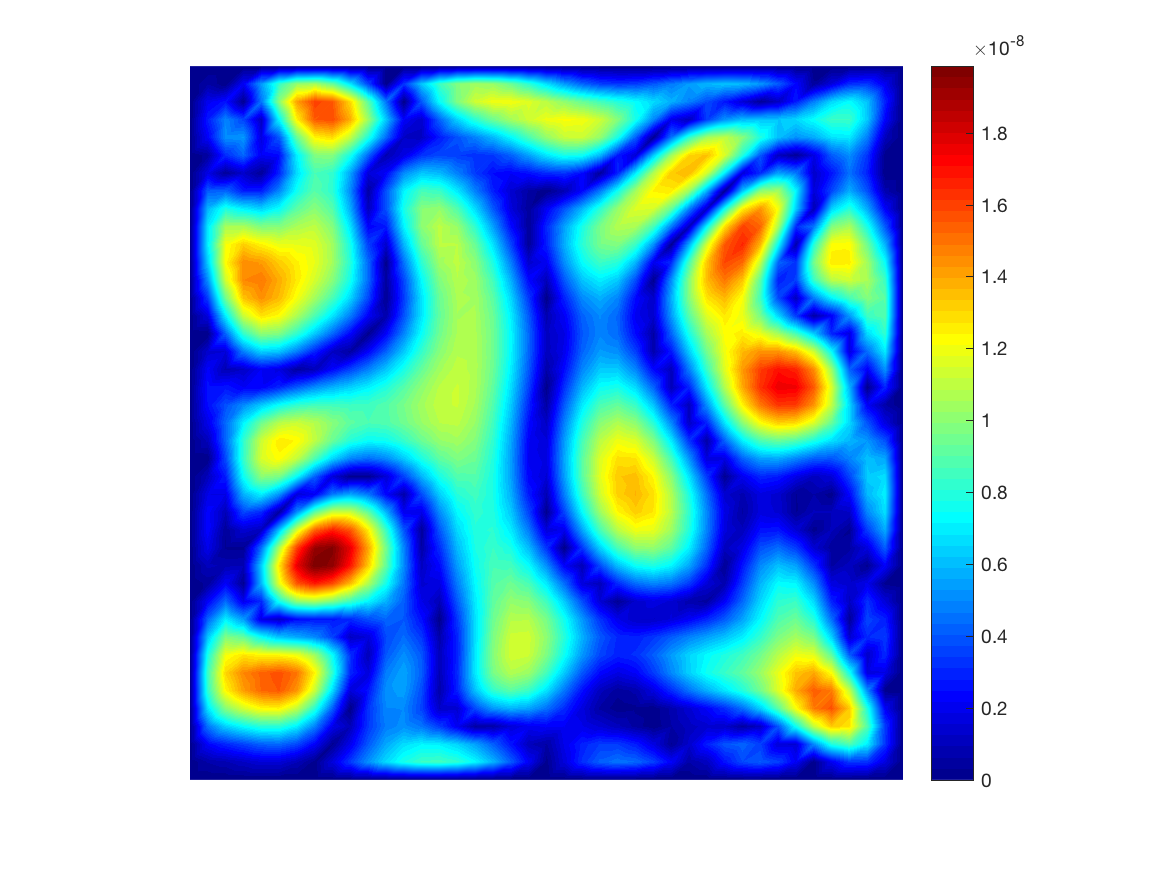}
\includegraphics[height=3cm]{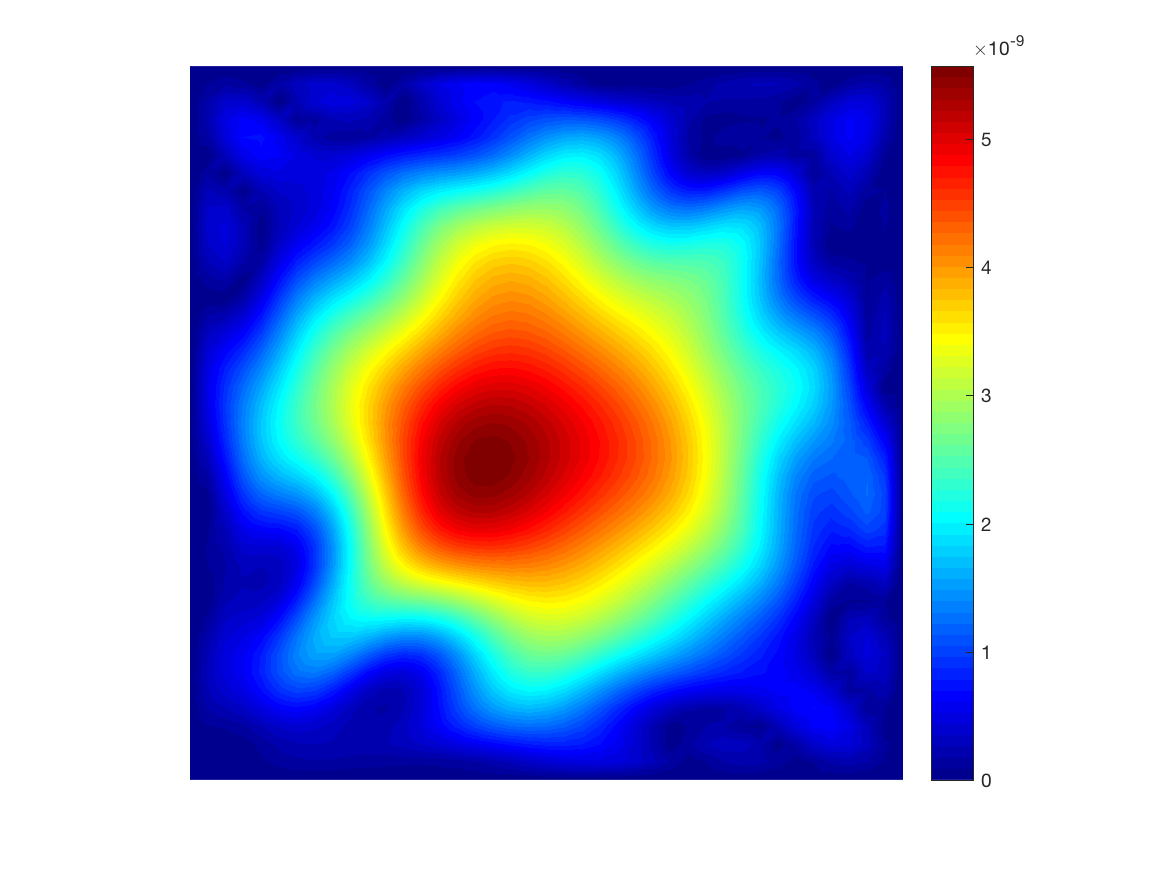}
\includegraphics[height=3cm]{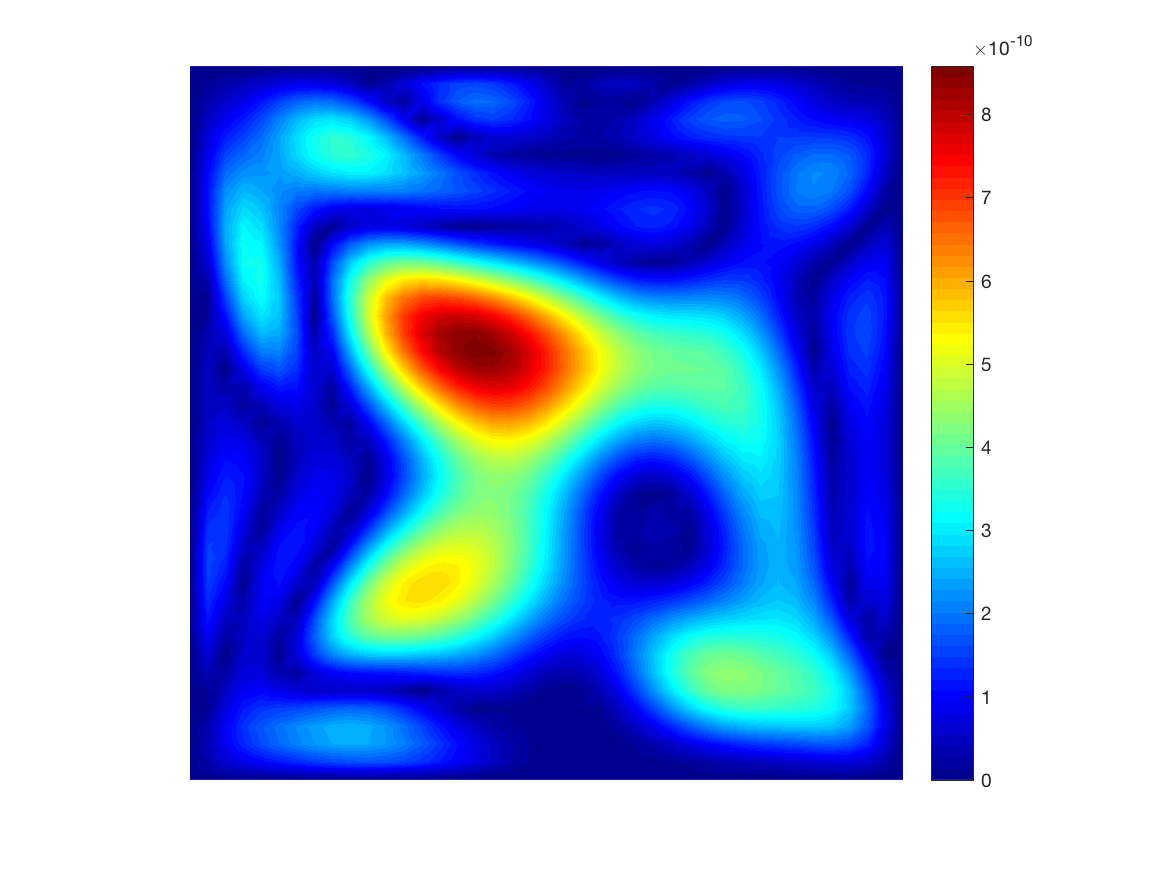}\\
 \includegraphics[height=3cm]{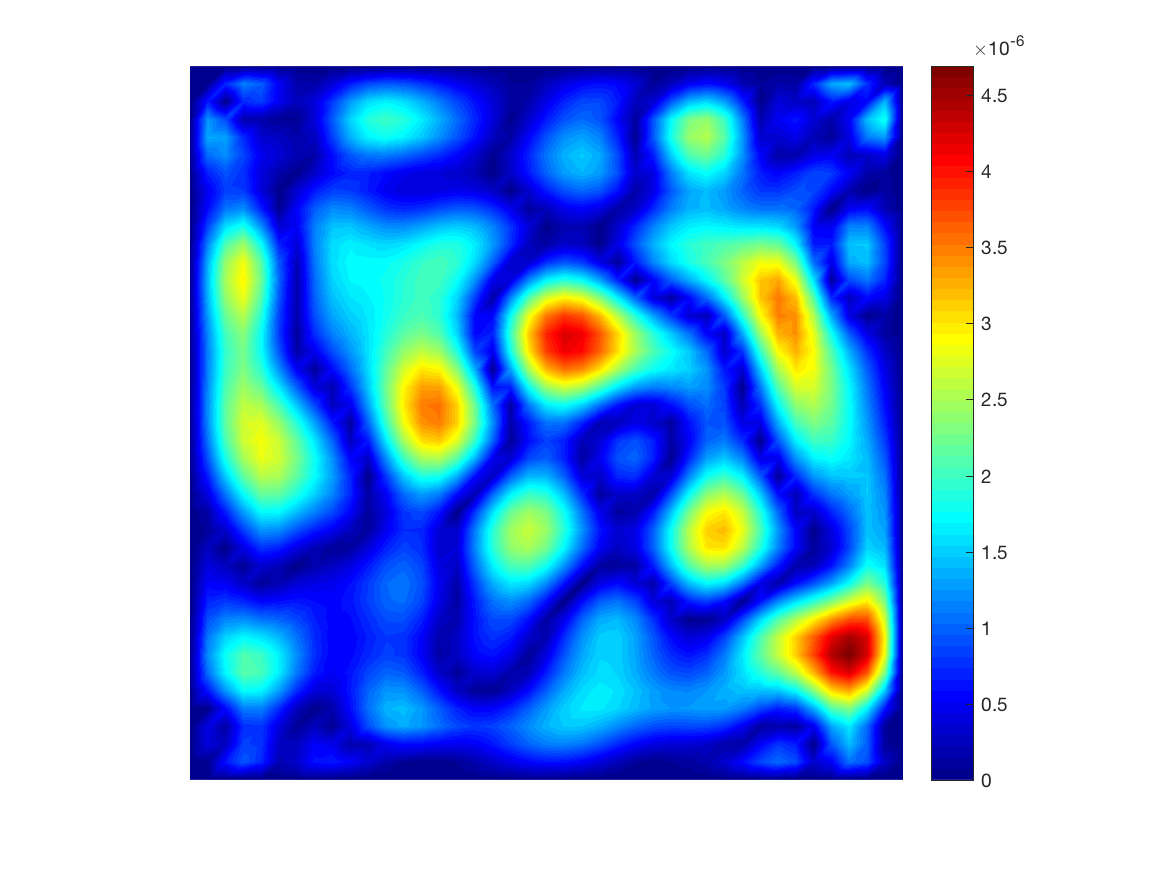}
\includegraphics[height=3cm]{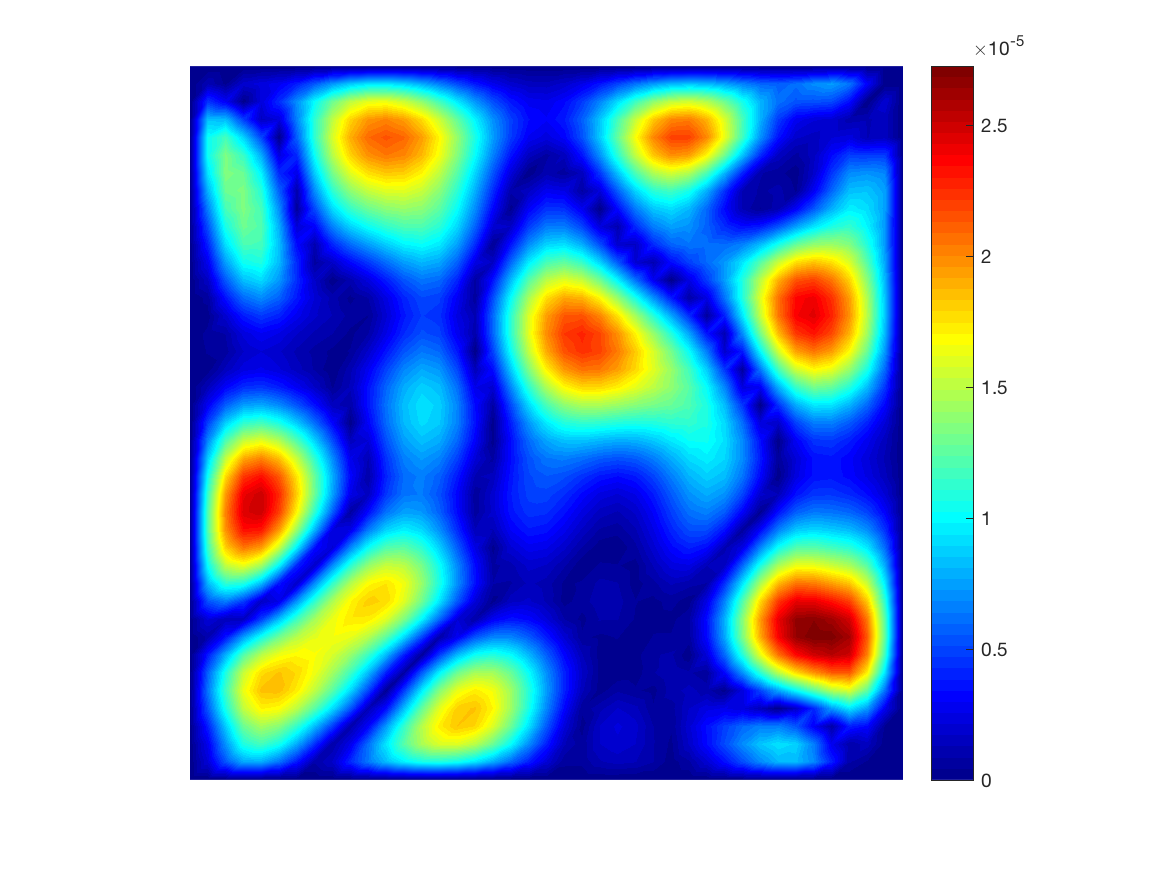}
\includegraphics[height=3cm]{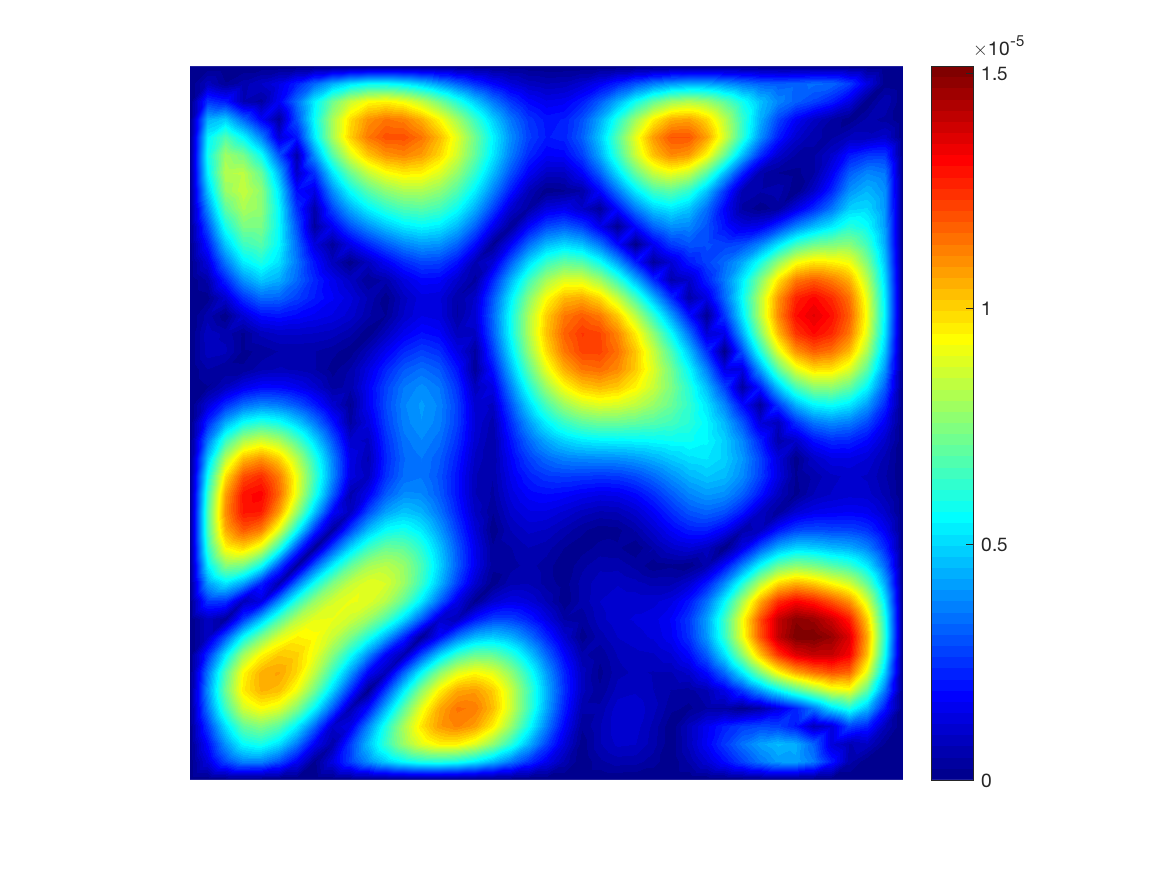}
\caption{Test case 2: absolute error between the exact solution and reduced approximations $|\Vu-\Vu_r|$
provided by MTD (top) and MTI (bottom) at final time for $r=30$ and $\Vxi \in\{ (0.2,-1),(0.2,0),(0.2,0.5)\}$.}
 \label{fig:ADVDIF2}
 \end{figure}
 
 }
 
 {
 Estimations of the expectations $\Exp(E_q(\Vxi))$ and maxima $\max_{\Vxi \in \VXi}(E_q(\Vxi))$  of the relative errors $E_q$ (using a random sample of size $50$ in $\VXi$) are plotted on Figure \ref{fig:ADVDIF3} for different values of $r$. For this advection-diffusion problem, we first observe that both order reduction methods provide accurate approximations with low-dimensional  spaces.  However, for the same dimension of the reduced spaces, the approximation obtained by MTD is more accurate than the approximation obtained by MTI. Indeed, in order to reach a relative error of $10^{-8}$, MTD requires $r=10$ while MTI requires $r=30$. 
 
\begin{figure}[H]
\centering
\includegraphics[height=5cm]{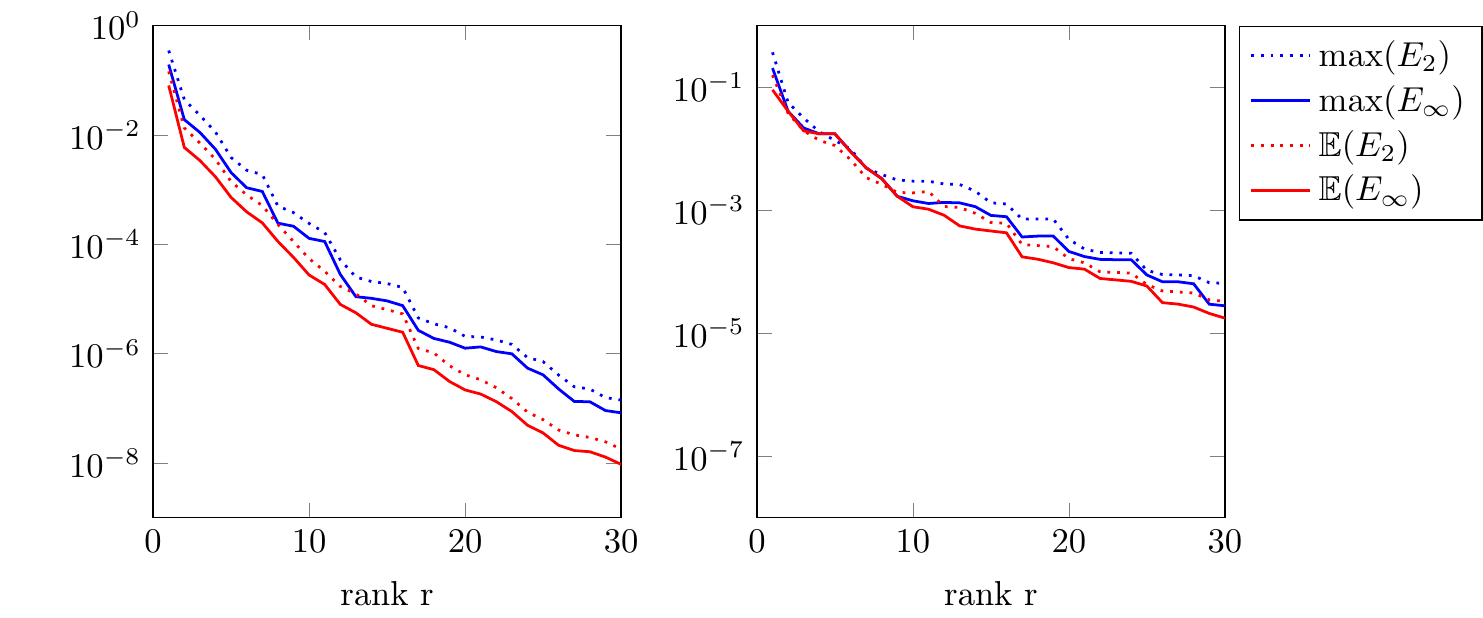}
 \caption{Test case 2:  Statistical estimations of the expectation and maximum of the relative errors $E_2$ and $E_\infty$   with respect to the reduced dimension $r$, for MTD (left) and MTI (right).\label{fig:ADVDIF3}}
\end{figure}
}

%  \mytitle{Computation with fixed rank $r$}
% \begin{center}
% {\scriptsize
% \begin{tabular}{|c|c|c||c|c|}
% \hline
%  &\multicolumn{2}{c||}{$E_2$} & \multicolumn{2}{c|}{$E_\infty$}\\
% \hline   
% r & T-G & POD-G & T-G& POD-G\\
% \hline 
% $5$   & 1.4017e-05   & 2.9247e-05   & 6.6568e-06    & 0.00098945  \\
% $10$ &  1.42e-08  & 2.1173e-05  &  3.2513e-08  & 3.4804e-05  \\
% $15$ &  2.6724e-12  &9.4355e-07  &6.1477e-12   &  2.7311e-06 \\
% $20$ &   3.0783e-15  &  5.9998e-08  & 5.5995e-15   &1.6498e-07  \\
% \hline
% \end{tabular}

{\paragraph{Error estimates}Figure \ref{fig:ADVDIF4} represents the  evolution with time of the exact  error $\|\Ve_r(t,\Vxi)\|_X$ and of the error estimate $\tilde \Delta_r(t,\Vxi)$ for both order reduction methods,  for $\Vxi=(0.2,0.5)$. Again, we observe the exponential behavior of the error estimates. The errors  obtained by MTD are several orders of magnitude lower than the errors obtained by MTI. 
\begin{figure}[H]
 \centering
  \includegraphics[height=5cm]{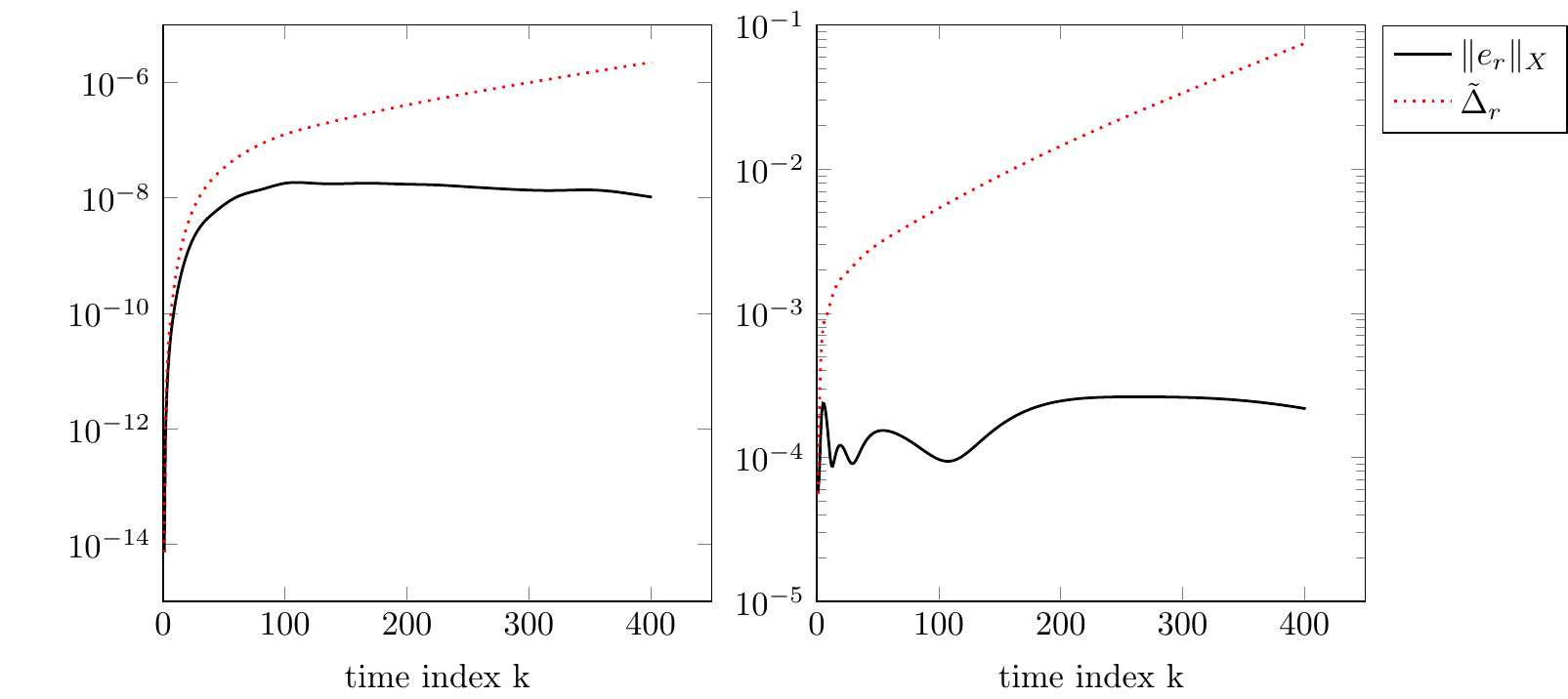}\\[0.1cm]
  \caption{Test case 2: Evolution with time of the error estimate $\tilde \Delta_r$ and true error $\|\Ve_r\|_X$ for  MTD (left) and MTI (right), for $r =30$ and $\Vxi = (0.2,0.5)$. \label{fig:ADVDIF4}}
 \end{figure}
}

{For the first time indices (i.e. $k<100$), $\tilde \Delta_r$ provides a sharper error bound for MTD than for MTI. Also, $\tilde \Delta_r$ and the true error have quite similar time evolutions for MTD, and very different time evolutions for MTI.
The superiority of MTD over MTI  is confirmed  onFigure \ref{fig:ADVDIF5}, where we observe that the expectations of the effectivity index $\kappa(t,\Vxi)$ of MTI and MTD are in a ratio of 10. For longer times, the  error estimates present similar efficiency for both approaches.
\begin{figure}[H]
 \centering
  \includegraphics[height=5cm]{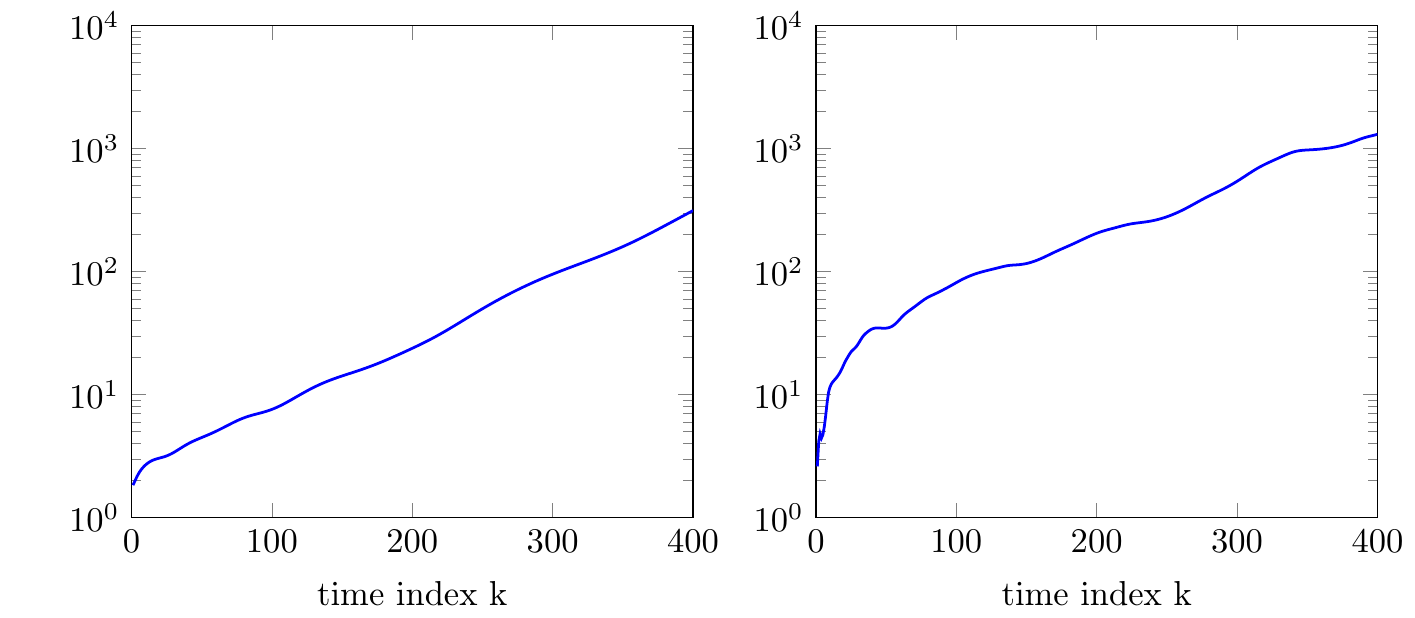}\\[0.1cm]
  \caption{Test case 2: Evolution with time of a statistical estimation of the expected effectivity index $\Exp(\kappa(t,\Vxi))$
  for MTD (left) and MTI (right). \label{fig:ADVDIF5}}
 \end{figure}
 }

{\paragraph{Study of greedy algorithms}
Now, we study the behavior of the  T-greedy and POD-greedy algorithms used for the construction of reduced  spaces for  MTD and MTI respectively. For both algorithms, the evolution with the iteration $r$ of the maximum of  $\Delta_r^{(0,T)}(\Vxi)$  over the training set $\VXi_{train}$ is plotted onFigure \ref{fig:ADVDIF6}. Note that at iteration $r$, the dimension  of the reduced space is actually $r+1$. 
We observe that the error decreases faster with $r$ for the T-greedy algorithm than for 
the POD-greedy algorithm. For example, at iteration $20$ of the greedy algorithms, 
the errors for POD-greedy and T-greedy algorithms are respectively of order
$10^{-1}$ and $10$. When stopping the greedy algorithms at a given precision, the dimension of the obtained reduced space is much smaller for T-greedy than for POD-greedy.  We observe on  Table \ref{tab:ADVDIF3} that the POD-greedy algorithm can select several times the same point in $\VXi$ (for example, the value $\Vxi=(-0.3192,0.9004)$ is
selected 6 times during the $30$ iterations). This  is due to the fact that at each iteration, only the first term of the POD  of the selected trajectory is added to the reduced space  (here $\ell=1$). Concerning the T-greedy approach, a point in $\VXi$ can be selected only once. This implies that for a fixed dimension, the POD-greedy approach requires less solutions 
of the full order dynamical system (here $5$ for POD-greedy against $30$ for the T-greedy approach when considering a reduced space of dimension $30$).} 
To illustrate this point, the number of calls to the full order problem with respect to $\max_{\xi \in \VXi_{train}}\Delta_r^{(0,T)}(\Vxi)$ is represented on the right curve ofFigure \ref{fig:ADVDIF6}. The number of calls to the full order problem for POD-greedy algorithm  remains smaller, which means that the POD-greedy algorithm uses only the information of snapshots already computed to enrich $X_r$ whereas the T-greedy algorithm requires new snapshots at each iteration. In terms of computational costs, the POD-greedy stategy presents lower offline computational costs since it requires a smaller number of calls to the full order  problem, but for a given precision, it leads to larger reduced systems in the online step. For the T-greedy algorithm, we have the opposite conclusions. For a given precision, it leads to lower dimensional reduced spaces, and therefore to lower online costs, but it requires more calls to the full-order problem in the offline step.
The MTD approach with POD-greedy remains a good alternative when treating autonomous linear dynamical systems arising from the discretization of linear parabolic problems with time independent matrices. Furthermore, contrarily to MTI, it does not require to store sequences of time dependent reduced quantities. But when considering non-autonomous dynamical systems (which requires handling time-dependent reduced quantities), the storing costs will be smaller with MTD, since this latter approach provides smaller reduced spaces for a given precision.

 \begin{figure}[H]
 \centering
 \includegraphics[height=5cm]{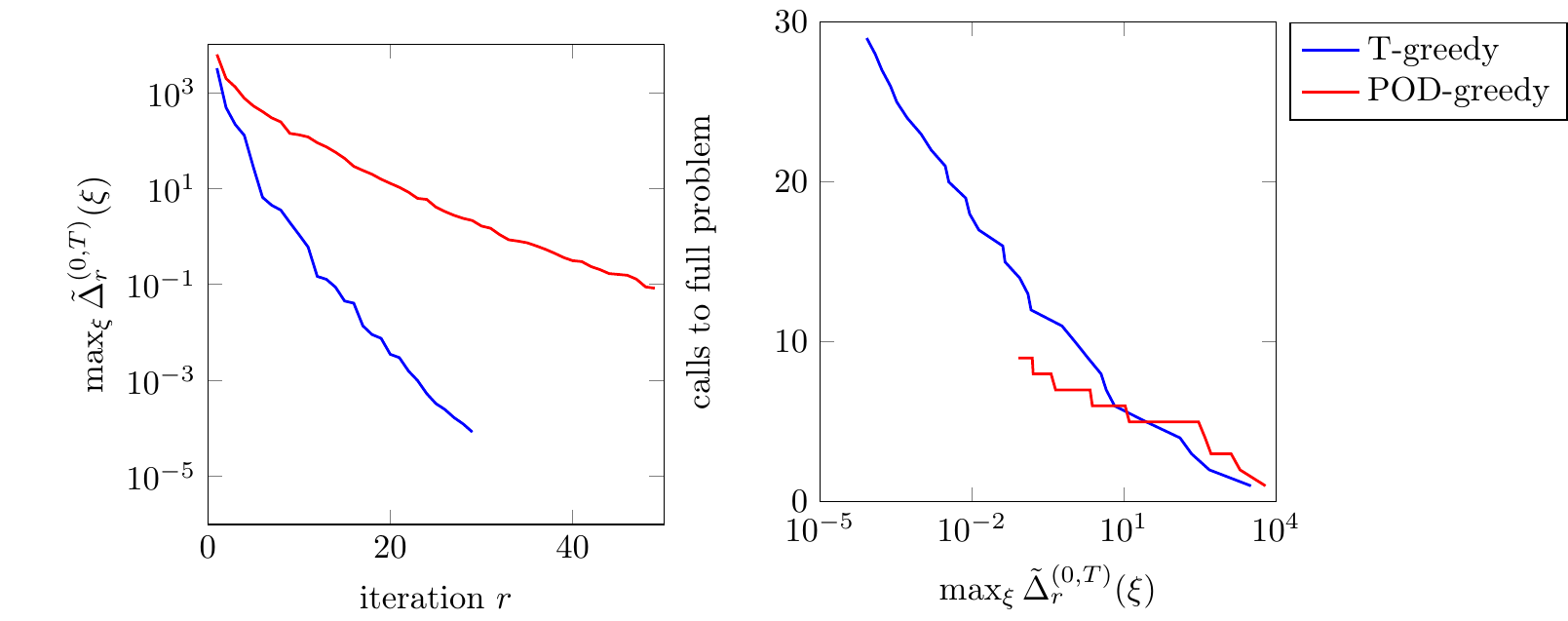}
\caption{Test case 2 : Evolution of the maximum of  $ \tilde \Delta_r^{(0,T)}(\Vxi)$ over $\VXi_{train}$ with the iteration $r$  
(left), and number of calls to the full order problem with respect to the maximum of  $\tilde \Delta_r^{(0,T)}(\Vxi)$ over $\VXi_{train}$ (right), for T-greedy (blue line) and POD-greedy (red line) algorithms. } \label{fig:ADVDIF6}
 \end{figure}
 
\begin{figure}[H]
 \centering
 \includegraphics[height=5cm]{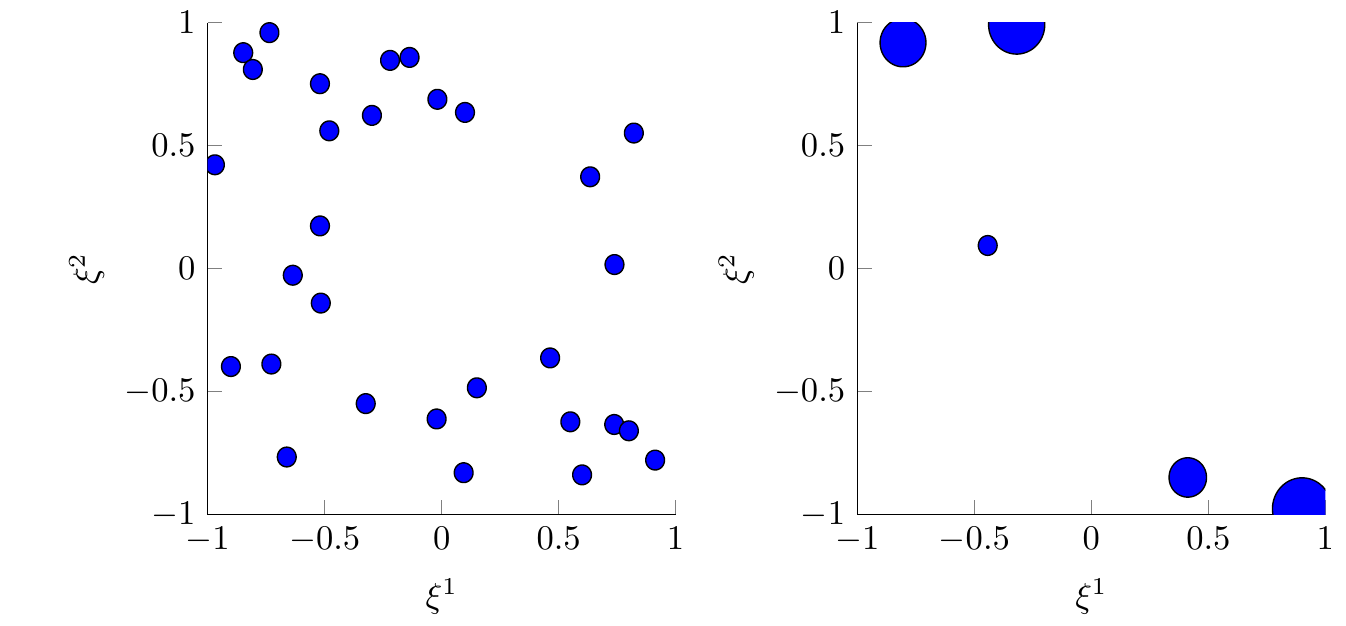}
 \caption{Test case 2: Selected points in $\VXi$ for the T-greedy (left) and  POD-greedy (right) algorithms during $30$ iterations. For POD-greedy, the size of  a point  increases with the number of times it is selected.}\label{tab:ADVDIF3}
\end{figure}

 {
 \paragraph{CPU times} 
As shown in Table \ref{tab:ADVDIFFCPU}, CPU-times for both offline and online phases are again similar, when the reduced space dimension is fixed to $r=30$ for both MTI and MTD. We still observe a small additional cost for the online phase of MTD due to computations with time-dependent reduced quantities, but  MTD  again provides  a better reduced approximation (see Figure  \ref{fig:ADVDIF3}). Concerning the POD-greedy algorithm, one could imagine to improve the offline cost by storing the  solution trajectories computed during the greedy procedure for parameters selected several times, but this would  imply additional storage costs.}
\begin{table}[H]
 \centering
% discontinu
  \begin{tabular}{|c|c|c|c|}
  \hline
& Offline - Basis selection & Offline - Reduced quantities & Online \\ %& Offline & Online\\
\hline   \hline
T-Greedy & $34907$ & $77.69$ & $37.86$\\% & POD-Greedy & $1756.8$ & $9.05$\\
POD-Greedy & $32767$  & $0.22$ & $34.6$\\%& T-Greedy & $ 2053.4$  &  $13.51 $\\
   \hline
\end{tabular}
  \caption{Test case 2: Offline and online CPU-times for both MTD and MTI.} \label{tab:ADVDIFFCPU}
\end{table}

\subsection{Test case 3}

We study a nonlinear viscous Burger's equation with uncertain parameters. This test case is adapted from \cite{Wirtz2014}. We consider a spatial domain $\Omega = (0,1)$, a time interval $I=(0,1)$, 
homogeneous Dirichlet boundary conditions and an initial condition $u^0=0$. We consider 
 $c(u,\Vxi) = \frac{1}{2}u$, $a(x,\Vxi)=0$ and a diffusion coefficient  $\mu(\Vxi) = \xi$,  with
  $\xi \sim U(0.01,0.06)$ a uniform random variable. The source term is defined by $g(x,t,\Vxi) = g_1(x,t) +  g_2(x,t)$, with 
\begin{align*}
&g_1(x,t) = 4e^{-(\frac{x-0.2}{0.03})^2}\Ind_{[0.1,0.3]}(x)\sin(4\pi t), \\
&g_2(x,t) = 4\cdot \Ind_{[0.6,0.7]}(x)\Ind_{[0.2,0.4]}(t).\\
\end{align*}
The term $g_1$ 
corresponds to an excitation localized in space and oscillating with time, 
whereas the term  $g_2$ corresponds to a constant excitation over a localized space-time region.
We use a finite difference scheme in space (with $n= 300$ nodes, and $d=298$) with uniform mesh (as in Subsection \ref{subsec:t1}) and a semi-implicit Euler scheme in time (with $K=200$ time steps), yielding the following scheme:
  \begin{equation}
(\MI_X - \delta t \MA(\Vxi)) \Vu^{k+1}(\Vxi)= \Vu^k(\Vxi)+ \delta t (\Vh(\Vu^k) + \Vg^k), \label{eq:pb14}
\end{equation}
where $ \MA \in \RR^{d \times d}$ denotes the discrete second derivative operator,
 and where the discrete flux $ \Vh(\Vu^k) \in \RR^d $ 
is defined by $(\Vh(\Vu^k))_i = - u_i^k (\MC \Vu^k)_i$ with $\MC \in \RR^{d \times d}$ the matrix associated with the discrete first derivative operator.\\

Reduced spaces are constructed with POD-greedy and T-greedy algorithms with a training set of size $60$.  For the error estimation, we have used a nearest neighbor interpolation of the Lipschitz constant of the flux $\Vh$. The maximum $M_L(\Vxi)$ and minimum $m_L(\Vxi)$ of $\frac{\tilde L _X [\Vh](\Vu_r(t,\Vxi),t,\Vxi)}{L _X[\Vh](\Vu_r(t,\Vxi),t,\Vxi)}$ over the time interval $I$ have been computed   for both MTD and MTI methods (see Table \ref{tab:BURGERS2}). We observe that the values are localized in intervals whose bounds are close to $1$,  which indicates that the approximations are quite satisfactory.  For the efficient evaluation of the nonlinear flux $\Vh$, an EIM has been used with a tolerance $\varepsilon =10^{-10}$ in the stopping criterion \eqref{eq:stoperrdeim}. For the considered simulations, {it corresponds to an average number $m$ of terms in the EIM equal to $103$ and $108$ for MTD and MTI respectively}.  As shown in  Table \ref{tab:BURGERS2}, the maximum of the relative errors $E_{\Vh}(t,\Vxi) = \|\Vh(u_r(t,\Vxi),t,\Vxi)- \tilde \Vh(u_r(t,\Vxi),t,\Vxi) \| _X/  \|\Vh(u_r(t,\Vxi),t,\Vxi) \|_X$ on the flux is of order $10^{-10}$ as expected.

   \begin{table}[H]
\centering
\begin{tabular}{c||c|c|c|c|}
& \multicolumn{2}{|c|}{$\max_t E_{\Vh}(\Vxi,t)$} 
& \multicolumn{2}{c|}{$[m_L(\Vxi),M_L(\Vxi)]$} \\
\hline
$\Vxi$     & MTI & MTD  & MTI &MTD\\
\hline \hline
$0.01$     & $1.5550~10^{-10}$   &  $4.2208~10^{-10}$   &   $[0.9116, 1.1011]$ & $[ 1.1055, 1.1297]$\\
$0.035$   & $3.5803~10^{-10}$   &  $1.4400~10^{-11}$  &   $[0.7804,0.9536]$  &$[1.0062, 1.0483]$ \\
$0.06$     & $5.0374~10^{-10}$   &  $1.4201~10^{-11}$  &   $[0.7567,1.0009]$ &$[0.9947, 0.9989]$
\end{tabular}
\caption{Test case 3 : Maximum over the time interval of the relative approximation error of the flux $\Vh$ with EIM, and minimum and maximum values $m_L$ and $M_L$ of the ratio between approximate and exact Lipschitz constants of $\Vh$, for MTD and MTI.}
 \label{tab:BURGERS2}
\end{table}

We first study the behavior of the approximations provided by the MTD and MTI evaluated for the parameter values $\Vxi \in \{0.01,0.035,0.06\}$ and for $r=15$. The evolution of the solution of the reduced order model obtained with MTD is plotted on  Figure \ref{fig:BURGERS21}. We clearly observe different features for different values of the diffusion coefficient.Figure \ref{fig:BURGERS22} represents the exact solution and the approximations obtained by MTD and MTI at final time and for a given value of the parameter. For $r=15$, the reduced approximations are in good agreement with the exact solution.  Nevertheless, we note that the approximation obtained with MTD is clearly better than with MTI.

 \begin{figure}[H]
 \centering
$\xi=0.01$ \hspace{3cm} $\xi = 0.035$  \hspace{2.5cm} $\xi=0.06$\\
 \includegraphics[height=3.1cm]{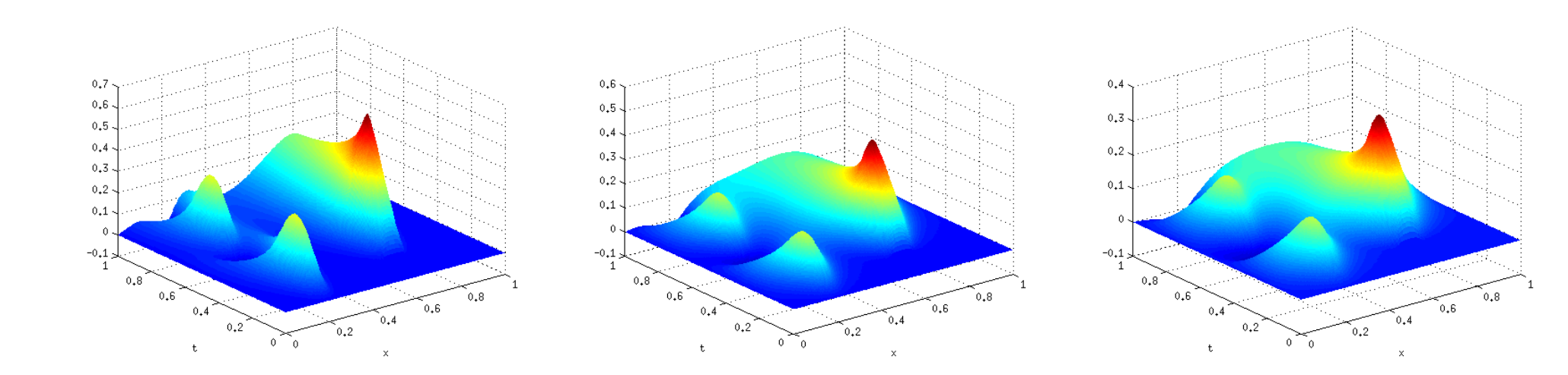}
\caption{Test case 3 : Evolution over the space-time grid of the approximation $\Vu_r$ computed with MTD for $r=15$ and 
 for $\Vxi \in \{0.01,0.035,0.06\}$.}
 \label{fig:BURGERS21}
 \end{figure}
 
  \begin{figure}[H]
 \centering
 \includegraphics[height=3.5cm]{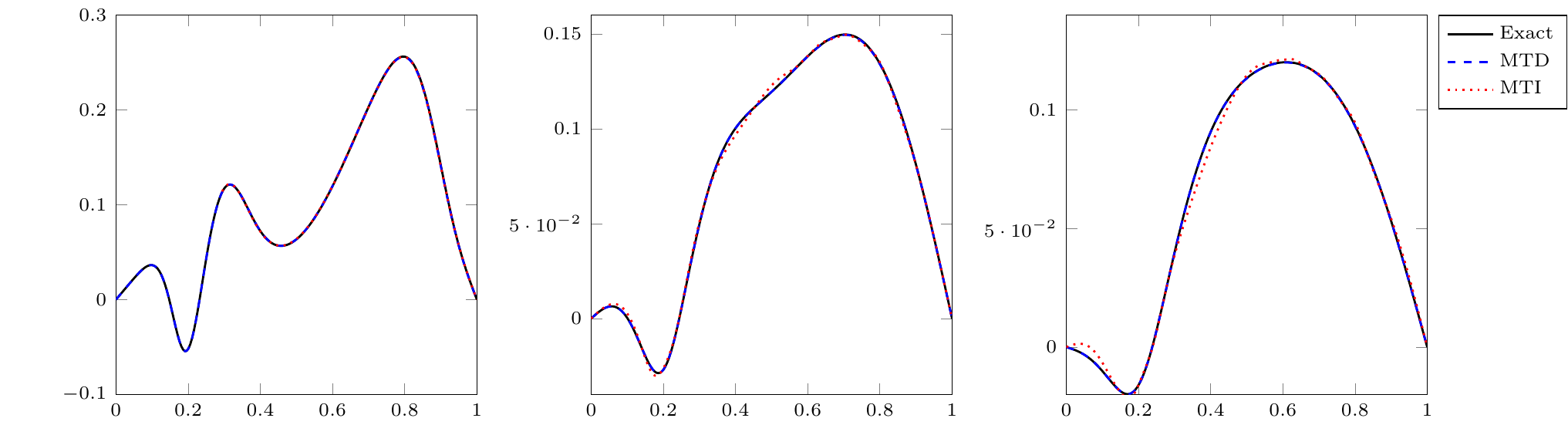}
\caption{Test case 3 : Comparison between the exact solution $\Vu$ (black line) and the approximations $\Vu_r$
obtained with  MTD (blue dashed line) and MTI (red dotted line) with $r=15$, at final time $T=0.1$ and for $\Vxi \in \{0.01,0.035,0.06\}$.}
 \label{fig:BURGERS22}
 \end{figure}

This superiority of MTD over MTI  is confirmed byFigure \ref{fig:BURGERS4} where we have plotted statistical estimations (with a random sample of size $50$) of the expectation and maximum of the relative errors. As we can see, MTD provides with $r=15$ an approximation with a relative error of $10^{-10}$, whereas  MTI   only provides an approximation with a relative error of $10^{-2}$ for the same dimension of the reduced space.  For $r=50$, MTI  provides an approximation with relative error $10^{-6}$, which is still higher than MTD with $r=15$.

 \begin{figure}[H]
\centering
\includegraphics[height=5cm]{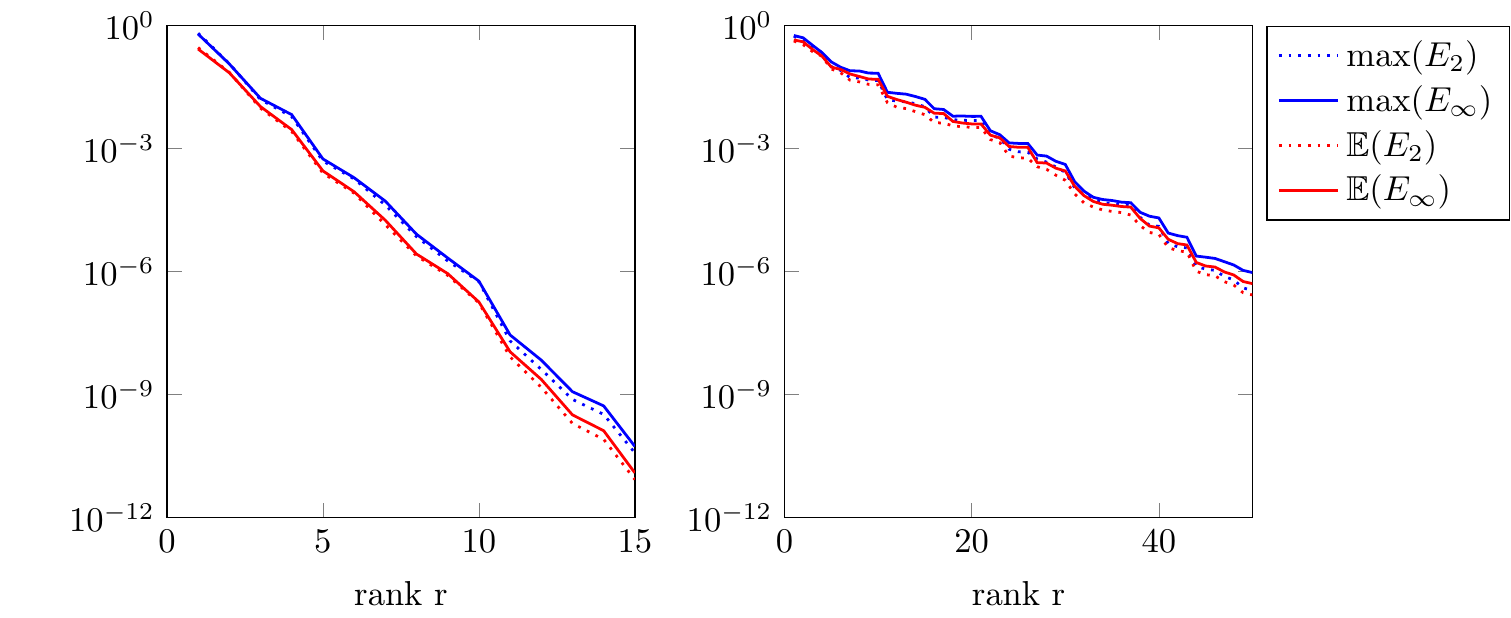}\\
 \caption{Test case 3:  Statistical estimations of  the expectations and maxima of the relative errors $E_2$ and $E_\infty$ with respect to the reduced dimension $r$, for MTD (left) and  MTI (right).}\label{fig:BURGERS4}
\end{figure}

\paragraph{Study of error estimates}

We now compare the behavior of the error estimate $\tilde \Delta_r$ for MTD and MTI. 
Figure \ref{fig:BURGERS5} represent the  evolution with time of the exact error and of the error estimate for both methods, with and without approximation of the flux and the Lipchitz constants, for $\Vxi=0.035$ and $r=15$. 
We first remark that the EIM approximation of the nonlinear flux as well as the approximated Lipschitz constants have a small impact on the behavior of the error estimate. Nevertheless we notice that for the first instants for MTD, where the error is smaller than $10^{-10}$, the EIM error is no longer negligible and  $\|\Ve_r\|_X$ and $\tilde \Delta_r$ are slightly impacted. For MTD, a small difference is observed for final times {only on} $\tilde \Delta_r$. This is probably due to the approximation of the Lipchitz constant. These differences disappear when choosing a smaller precision $\varepsilon$ for MTD or when considering larger reduced spaces for MTI. 
 Again, we obtain several orders of magnitude between the errors provided by the MTD and MTI. Here, the expectations of the effectivity index $\kappa(t,\Vxi)$ (plotted onFigure \ref{fig:BURGERS3}) are of the same order, with or without EIM, which means that the error estimate has the same efficiency for both approaches. 
\begin{figure}[H]
\centering
 \includegraphics[height=5cm]{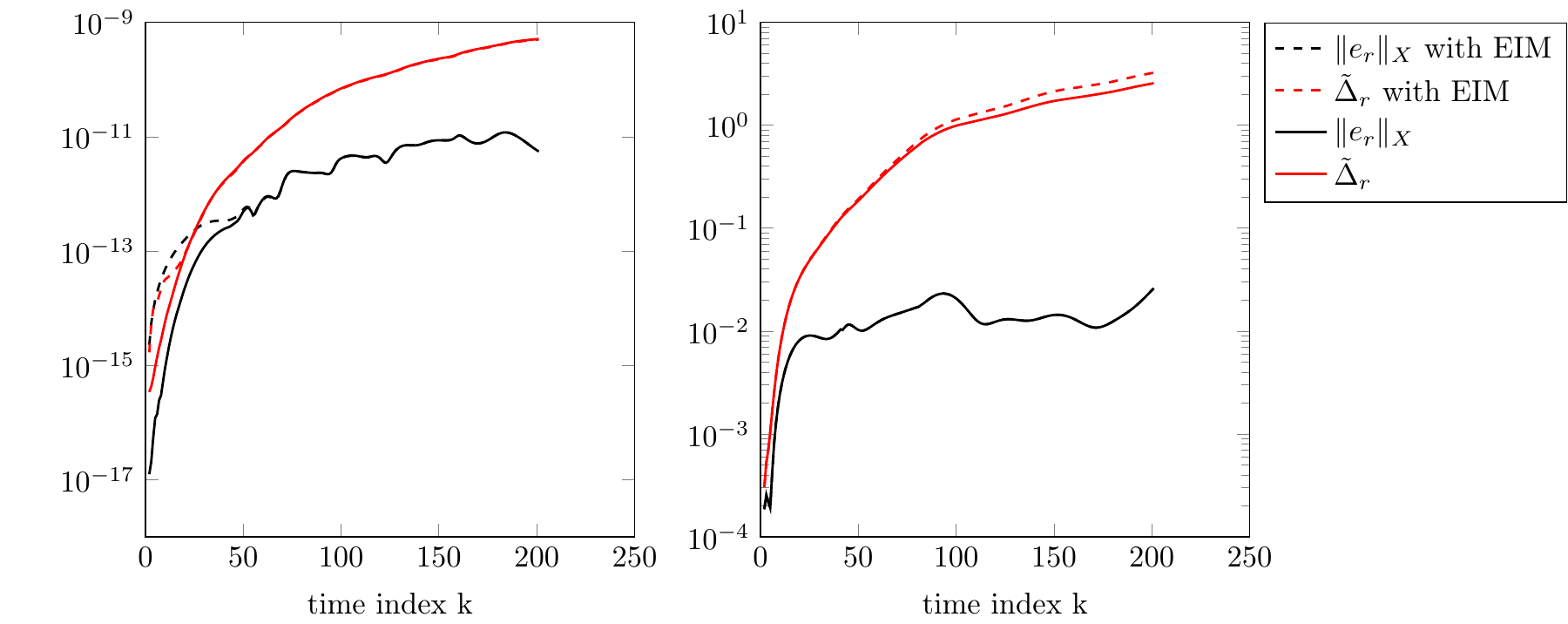}
   \caption{Test case 3:  Evolution with time of the error estimate  $\tilde \Delta_r$ (dotted line) and the exact error $\|\Ve_r\|_X$ 
  (solid line), for  MTD (left) and MTI (right), for $\Vxi =0.035$ and $r=15$.
  \label{fig:BURGERS5}}
\end{figure}

 \begin{figure}[H]
\centering
 \includegraphics[height=5cm]{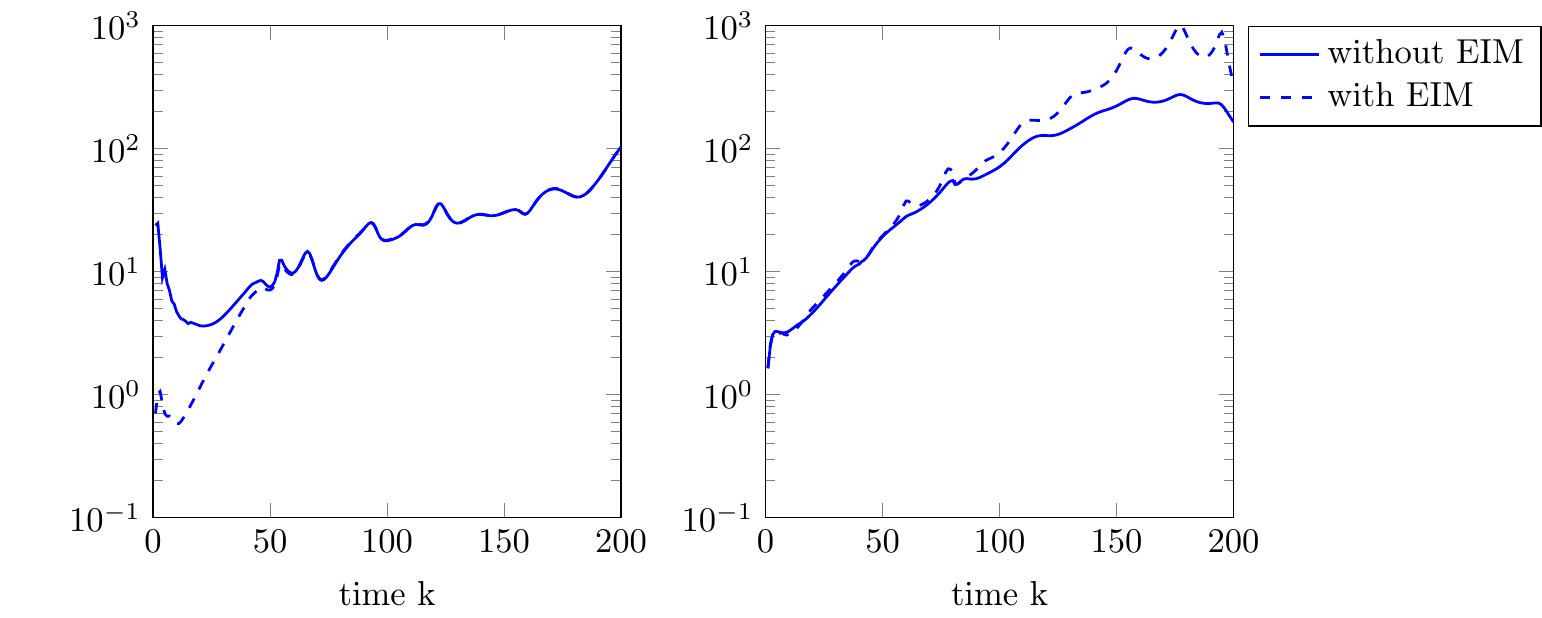}
  \caption{Test case 3:  Evolution with time of a statistical estimation of the expected effectivity index $\Exp(\kappa(t,\Vxi))$, for MTD (left) and MTI (right), for $\Vxi =0.035$ and $r=15$.\label{fig:BURGERS3}}
 \end{figure}
 
 \paragraph{Study of greedy algorithms}
 
 We finally study the behavior of the greedy algorithms used for the construction of the reduced space $X_r$. Once again, we have plotted 
  the evolution of the maximum of $ \tilde \Delta_r^{(0,T)}(\Vxi)$ over $\VXi_{train}$ with the iteration $r$  and the number of calls to the full order problem with respect to the maximum of $ \tilde \Delta_r^{(0,T)}(\Vxi)$ onFigure \ref{fig:BURGERS6}.  We observe the same behavior as for the test case 2. The maximum error during the greedy iterations decreases faster with T-greedy than with POD-greedy (T-greedy reaches a precision of order $10^{-6}$ for the iteration $14$, while   POD-greedy  reaches a precision of $10^{-2}$ for the iteration $49$). In terms of computational costs, the number of calls to the full order problem in the offline step is lower for  T-greedy algorithm than for  POD-greedy (e.g., for reaching a precision of order $10^{-2}$, T-greedy (resp. POD-greedy) algorithm requires  $7$ (resp. 11) calls to the full order problem). Contrary to the previous linear test case, the use of MTD with T-greedy algorithm clearly reduces the computational costs of both online and offline steps in comparison to a MTI combined with a POD-greedy construction of reduced spaces.  
  \begin{figure}[H]
 \centering
 \includegraphics[height=5cm]{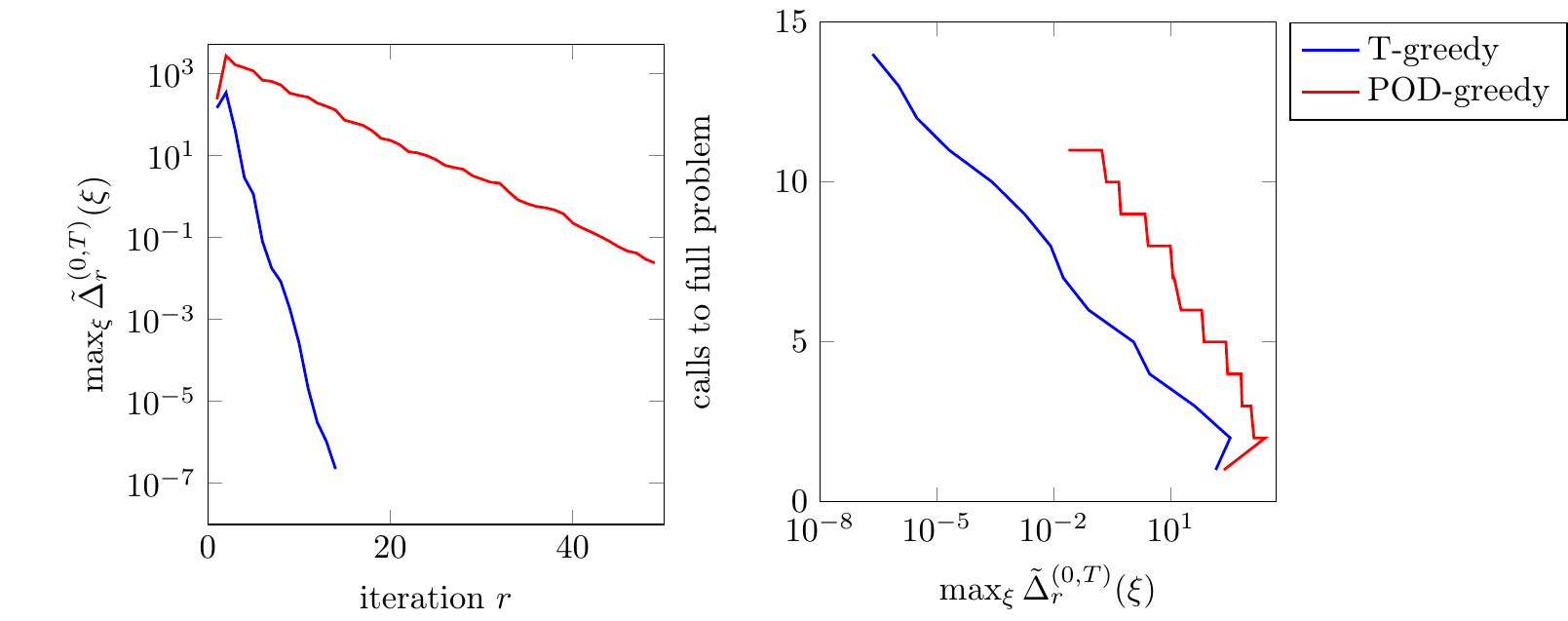}
\caption{Test case 3 : Evolution of the maximum of  $ \tilde \Delta_r^{(0,T)}(\Vxi)$ over $\VXi_{train}$ with the iteration $r$  
(left), and number of calls to the full order problem with respect to the maximum of  $\tilde  \Delta_r^{(0,T)}(\Vxi)$ over $\VXi_{train}$ (right), for T-greedy (blue line) and POD-greedy (red line) algorithms. } \label{fig:BURGERS6}
 \end{figure}

 {
 \paragraph{CPU times} 
 The dimension of reduced spaces is fixed to $r=15$ for both methods.
As summarized in the table Table \ref{tab:BURCPU}, once again for this problem, both offline and online cost are roughly of the same order for both MTI and MTD while MTD gives a finer reduced approximation which is by 8 orders of magnitude (see Figure \ref{fig:BURGERS4}) better than the one provided by MTI .} %Moreover, we do not notice any particular impact of the EIM procedure for reduced flux computation that is similar for both approaches.}
\begin{table}[H]
 \centering
% discontinu
  \begin{tabular}{|c|c|c|}
%   \multicolumn{3}{c}{With $u_{disc}^0$} &  \multicolumn{3}{c}{With $u_{cont}^0$}\\
   \hline
& Offline (Basis) & Offline (Reduction) \\ %& Offline & Online\\
\hline   \hline
MTI   & $6137.2$ &  $9.22$\\% & POD-Greedy & $1756.8$ & $9.05$\\
MTD & $6253.8$  &  $9.27$\\%& T-Greedy & $ 2053.4$  &  $13.51 $\\
   \hline
\end{tabular}
  \caption{Test case 1: Offline and online CPU-times for both MTD and MTI. \label{tab:BURCPU}}
\end{table}

 %%%%%%%%%%%%%%%%%%%%%%%%%%%%%%%%%%%%%%%%%%%
  % CONCLUSION
 %%%%%%%%%%%%%%%%%%%%%%%%%%%%%%%%%%%%%%%%%%%

\section{Conclusion}
In this paper, we have presented a projection-based model order reduction approach for the solution of parameter-dependent  
dynamical systems. It relies on a Galerkin projection of the full order dynamical system on time-dependent reduced spaces. This generalizes classical model order reduction methods with time-independent reduced spaces.  An a posteriori error estimate using the logarithmic Lipschitz constant associated to the flux has also been proposed and provides an efficient a posteriori error estimate. Using this   error estimate, we  have derived a greedy algorithm (called T-greedy algorithm) for the construction of a  sequence of  time-dependent reduced spaces. 
We have performed several numerical tests on both  linear and nonlinear dynamical systems. 
The order reduction method with time-dependent spaces constructed with a T-greedy algorithm (MTD) 
has been compared to a method with time-independent spaces constructed with a POD-greedy algorithm (MTI).
%We have demonstrated that for a given precision, MTD always yields a lower dimensional reduced order model than MTI. For non-autonomous dynamical systems, for which the reduced order models depend on time-dependent quantities even for time-independent spaces, MTD therefore yields lower online computational costs than MTI. 
%{For  autonomous linear dynamical systems (e.g., coming from the spatial discretization of linear parabolic equation), 
%MTI has smaller offline costs cost than MTD (for a fixed rank $r$) since it requires a lower number of calls to the full order model and it does not require the storage of time dependent reduced matrices. }
%However, for non-autonomous and nonlinear dynamical systems, MTD outperforms MTI in both offline and online steps. 
{For the same online costs (time for evaluating the ROM), the error in the predictions of the ROM obtained by MTD may be several orders of magnitude lower than the error in the predictions of the ROM obtained by MTI. In other words, for a given precision of the ROM, the MTD approach yields a significant reduction of online computation costs.}
%{Regarding CPU times, the methods perform in a similar way but MTD requires additional computation and storage efforts due to the computation of time dependent reduced quantities when dealing with linear and autonomous dynamical systems.}
Here, the proposed T-greedy algorithm for the selection of parameters relies on an error indicator using a $L^2$-norm in time of an error estimate.
{This global error estimate could be improved by considering sharper local in time error estimates which could be obtained by improving the local Lipschitz constants used to derive our error bound \cite{Wirtz2014}.} Moreover, a goal-oriented variant of this algorithm could be introduced by considering  norms taking into account a certain quantity of interest (e.g., the value of the solution at the final time) in order to improve the efficiency of the order reduction method for this approximation of this quantity of interest.

% BIBLIO
%%%%%%%%%%%%%%%%%%%%%%%%%%%%%%%%%%%%%%

\end{document}